\documentclass[11pt, a4paper, reqno]{amsart}

\usepackage{amscd}
\usepackage{dsfont}
\usepackage{amsmath, upgreek}
\usepackage{amstext}
\usepackage{amsthm}
\usepackage{amssymb}
\usepackage{amsopn}
\usepackage{amssymb}
\usepackage{amsxtra}
\usepackage{amsfonts}
\usepackage{fancyhdr}
\usepackage{paralist, epsfig}
\usepackage[mathcal]{euscript}
\usepackage[english]{babel}
\usepackage[latin1]{inputenc}
\usepackage{bbold}
\usepackage{setspace}
\usepackage[svgnames]{xcolor}
\usepackage{mathtools}
\usepackage{array, tabu}

\usepackage{tikz}
\usepackage{tikz-cd}
\usepackage{tkz-euclide}
\usepackage{euscript}
\usepackage{hyperref}
\hypersetup{
    colorlinks = true,
    linkcolor = RoyalBlue,
    citecolor = Green,    
}

\newcommand{\qcomp}[1]{\overbracket[0.6pt][1.5pt]{#1\hspace{1pt}}}
\newcommand{\hc}{\operatorname{hc}\nolimits}

\renewcommand{\emptyset}{\varnothing}
\newcommand{\T}{\operatorname{\scriptscriptstyle\mathsf T}}

\makeatletter
\g@addto@macro\normalsize{%
  \setlength\abovedisplayskip{10pt}
  \setlength\belowdisplayskip{10pt}
  \setlength\abovedisplayshortskip{5pt}
  \setlength\belowdisplayshortskip{8pt}
}

\newtheorem{thm}{Theorem}[section]
\newtheorem{cor}[thm]{Corollary}
\newtheorem{lem}[thm]{Lemma}
\newtheorem{prop}[thm]{Proposition}

\theoremstyle{definition}
\newtheorem{ex}[thm]{Example}
\newtheorem{rmk}[thm]{Remark}
\newtheorem{dfn}[thm]{Definition}

\renewcommand{\epsilon}{\varepsilon}

\newcommand{\Hom}{\operatorname{Hom}\nolimits}

\newcommand{\inc}{\operatorname{inc}\nolimits}
\newcommand{\cone}{\operatorname{cone}\nolimits}

\newcommand{\hMon}{\operatorname{hMon}\nolimits}
\newcommand{\hEpi}{\operatorname{hEpi}\nolimits}

\newcommand{\hMork}{\ensuremath{\operatorname{hMor}_{\operatorname{ker}}\nolimits}}
\newcommand{\hMorc}{\ensuremath{\operatorname{hMor}_{\operatorname{cok}}\nolimits}}

\newcommand{\spann}{\operatorname{span}\nolimits}

\newcommand{\id}{\operatorname{id}\nolimits}

\newcommand{\Bigprod}[2]{\ensuremath{\mathop{\textstyle\prod}_{#1}^{#2}}}
\newcommand{\ran}{\operatorname{ran}\nolimits}
\newcommand{\cok}{\operatorname{cok}\nolimits}
\newcommand{\coim}{\operatorname{coim}\nolimits}
\newcommand{\im}{\operatorname{im}\nolimits}

\newcommand{\epi}{\rightarrowtriangle}

\newcommand{\Der}{\ensuremath{\operatorname{\mathsf{D}}}}
\newcommand{\K}{\ensuremath{\operatorname{\mathsf{K}}}}
\newcommand{\C}{\ensuremath{\operatorname{\mathsf{Ch}}}}
\newcommand{\Ac}{\ensuremath{\operatorname{\mathsf{Ac}}}}
\newcommand{\bb}{\ensuremath{\operatorname{\mathsf{b}}}}

\newcommand{\LL}{\ensuremath{\operatorname{L}}}

\renewcommand{\mod}{\operatorname{mod}}
\newcommand{\NN}{\mathbb{N}}

\newcommand{\CC}{\mathbb{C}}
\newcommand{\DD}{\mathbb{D}}

\newcommand{\KK}{\mathbb{K}}
\newcommand{\EE}{\mathbb{E}}
\newcommand{\ZZ}{\mathbb{Z}}

\newcommand{\II}{\mathbb{I}}

\newcommand{\diam}{\hfill\raisebox{-8pt}{\scalebox{0.9}{\rotatebox{45}{$\qed$}}}}

\renewcommand{\bullet}{\,\begin{picture}(-1,1)(-1,-3)\circle*{2.}\end{picture}\ }

\newcommand{\COM}{\mathsf{COM}}
\newcommand{\HCOM}{\mathsf{HCOM}}

\newcommand{\Ban}{\mathsf{Ban}}
\newcommand{\Hilb}{\mathsf{Hilb}}

\newcommand{\LCS}{\mathsf{LCS}}
\newcommand{\Bor}{\mathsf{Born}}
\newcommand{\CBor}{\mathsf{CBorn}}
\newcommand{\SBor}{\mathsf{SBorn}}
\newcommand{\BMod}{\mathsf{BMod}}

\newcommand{\HDLCS}{\mathsf{LCS}_{\mathsf{HD}}}

\newcommand{\FreDN}{\mathsf{Fre}_{\mathsf{(DN)}}}

\newcommand{\Fre}{\mathsf{Fre}}

\newcommand{\hypref}[1]{\hyperref[R0]{\color{RoyalBlue}{\bf #1}}}

\usepackage{xstring}

\newcommand\zfrac[2]{\text{\footnotesize\raisebox{.15ex}{%
\dimen0=\fontdimen8\textfont2  
\dimen2=\fontdimen11\textfont2 
\dimen4=\fontdimen8\textfont3  
$%
\fontdimen8\textfont2=.5\dimen0
\fontdimen11\textfont2=.5\dimen2
\fontdimen8\textfont3=1.1\dimen4
\dfrac{#1}{#2}$%
\fontdimen8\textfont2=\dimen0
\fontdimen11\textfont2=\dimen2
\fontdimen8\textfont3=\dimen4
}}}

\renewcommand*\textcircled[1]{\tikz[baseline=(char.base)]{
            \node[shape=circle,draw,inner sep=1pt] (char) {#1};}}

	\usetikzlibrary{matrix,arrows,decorations,decorations.pathmorphing,positioning,decorations.pathreplacing,shapes}
	\tikzset{commutative diagrams/.cd, 
		mysymbol/.style = {start anchor=center, end anchor = center, draw = none}}

\newcommand{\commutes}[2][\circ]{\arrow[mysymbol]{#2}[description]{#1}}

\DeclareMathSymbol{\sm}{\mathbin}{AMSa}{"39}
\DeclareMathSymbol{\shortminus}{\mathbin}{AMSa}{"39}

\DeclareMathSymbol{\shortminus}{\mathbin}{AMSa}{"39}

\usetikzlibrary{arrows.meta}

\usepackage{enumitem}

\newlist{myitemize}{itemize}{1}
\setlist[myitemize,1]{leftmargin=26pt, noitemsep, topsep=0pt}

\renewcommand{\tau}{\uptau}
\renewcommand{\sigma}{\upsigma}
\renewcommand{\phi}{\upphi}
\renewcommand{\psi}{\uppsi}
\renewcommand{\alpha}{\upalpha}
\renewcommand{\beta}{\upbeta}
\renewcommand{\gamma}{\upgamma}
\renewcommand{\delta}{\updelta}
\renewcommand{\mu}{\upmu}
\renewcommand{\nu}{\upnu}
\renewcommand{\rho}{\uprho}
\renewcommand{\xi}{\upxi}
\renewcommand{\zeta}{\upzeta}
\renewcommand{\kappa}{\upkappa}

\usepackage{stmaryrd}
\usepackage[OT1]{fontenc}

\usepackage{pifont}

\begin{document}

\allowdisplaybreaks
$ $
\vspace{-40pt}

\title{Homological surrogates in topological\\[2pt]and bornological analysis}

\author{\phantom{xx}Marianne M.\ Lawson\hspace{0.5pt}\MakeLowercase{$^{\text{1},\,\text{2}}$} and\; Sven A.\ Wegner\hspace{0.5pt}\MakeLowercase{$^{\text{1},\,\text{2},\,\text{3}}$}}

\renewcommand{\thefootnote}{}
\hspace{-1000pt}\footnote{\hspace{5.5pt}2020 \emph{Mathematics Subject Classification}: Primary 46M18, 18G80; Secondary 18E35, 46M03, 46A17.\vspace{1.6pt}}


\hspace{-1000pt}\footnote{\hspace{5.5pt}\emph{Key words and phrases}: karoubian category, (one-sided) exact category, abelianization, derived equivalence,\newline\phantom{x}\hspace{1.2pt}quasi-complete spaces, sequentially complete spaces, locally complete spaces, bornological modules. \vspace{1.6pt}}

\hspace{-1000pt}\footnote{\hspace{0pt}$^{1}$\,University of Hamburg, Department of Mathematics, Bundesstra\ss{}e 55, 20146 Hamburg, Germany,\newline\phantom{x}\hspace{1.2pt}e-mail: $\{$marianne.lawson,\;sven.wegner$\}$@uni-hamburg.de.\vspace{1.6pt}}

\hspace{-1000pt}\footnote{\hspace{0pt}$^{2}$\,This research was funded by the German Research Foundation\,(DFG); Project no.~507660524.\vspace{1.6pt}}

\hspace{-1000pt}\footnote{\hspace{0pt}$^{3}$\,Corresponding author.}

\begin{abstract}We describe hearts of strongly deflation-exact categories with kernels as localizations of categories of three-term complexes up to homotopy, without assuming admissibility of kernels. We construct a calculus of fractions and give explicit formulas for kernels and cokernels. For complete locally convex spaces, whose failure to be quasiabelian was established by Prosmans (2000), our description and its dual provide derived equivalent abelian models through two choices of a notion of exactness. For quasi-complete, sequentially complete and locally (a.k.a.\ Mackey) complete spaces, we show that the topologically exact sequences form the maximal deflation-exact structure, although these categories are not extension-closed in the category of locally convex Hausdorff spaces. Our results make them nevertheless accessible to homological algebra with respect to the induced notion of exactness. Further applications concern the quasiabelian category of bornological modules equipped with the non-maximal linear split exact structure.
\end{abstract}


\maketitle

\vspace{-15pt}
\section{Introduction}\label{SEC-1}

Homological methods are a ubiquitous tool in algebra, geometry and analysis. In all three areas this may involve dealing with non-abelian categories; in particular this is inevitable when categories of topological or bornological vector spaces or modules appear. The fundamental concept which then comes into play is that of a Quillen exact category \cite{Quillen72} in which one declares a class $\CC$ of kernel-cokernel pairs in an additive category $\mathcal{A}$ as the short exact sequences. Provided that $\CC$ enjoys certain axioms and $\mathcal{A}$ satisfies some mild assumptions, it turns out that one can do homological algebra relative to $\CC$. This includes the definition of the derived category, see Neeman \cite{Nee90}.

\smallskip

Given a category $\mathcal{A}$ in which every morphism has a kernel as well as a cokernel and the class $\CC_{\operatorname{all},\,\mathcal{A}}$ of all kernel-cokernel pairs satisfies the aforementioned axioms, the category $\mathcal{A}\equiv(\mathcal{A},\CC_{\operatorname{all},\,\mathcal{A}})$ is called quasiabelian. It was shown by Schneiders \cite{S} that in this case the derived category $\Der(\mathcal{A})$ admits two natural t-structures whose hearts are abelian categories, referred to as the left and right heart, respectively. Schneiders also showed that there are natural embeddings $\mathcal{A}\rightarrow\mathcal{LH}(\mathcal{A})$ and $\mathcal{A}\rightarrow\mathcal{RH}(\mathcal{A})$ which lift to derived equivalences, and that the objects in the left heart can be represented by monomorphisms of the category $\mathcal{A}$. By direct construction this category had been defined even earlier by Waelbroeck \cite{W, WF, Wbor} for Banach, Fr\'echet and bornological vector spaces, see also Wegner \cite{Wegner17}.

\smallskip

Henrard et al.\ \cite{HKRW} generalized Schneiders' results by showing that if $\mathcal{A}$ has kernels, but need not have cokernels, and $\CC$ is a one-sided so-called strongly deflation-exact structure as considered by Rump \cite{RumpMax} and Bazzoni, Crivei \cite{BC13}, then the left heart exists and is derived equivalent to $(\mathcal{A},\CC)$. If $\mathcal{A}$ additionally has admissible kernels (in the sense that $(f,\cok f)\in\CC$ for every kernel $f$), then also the representation of the left heart via $\mathcal{A}$-monomorphisms holds. Dual statements are of course valid as well. The term `homological surrogate' in the title serves as an umbrella term for the situation just sketched and indicates that given a possibly one-sided exact category $(\mathcal{A},\CC)$, we search for a preferably abelian category $\mathcal{B}$ such that $\Der^{\bb}(\mathcal{A},\CC)\cong\Der^{\bb}(\mathcal{B})$ holds in a natural way. We will see that given $\mathcal{A}$ there might be several natural choices for $\CC$ leading to different surrogates. Related notions are homological embeddings, see Psaroudakis \cite{Psaroudakis}, Marks \cite{Marks17}, or Meyer's isocohomological morphisms \cite{Meyer3}.

\smallskip

In the first part of this article we focus on strongly deflation-exact categories in which every morphism has a kernel. We do \emph{not} assume the existence of cokernels for any morphism and we also do \emph{not} assume that all kernels are inflations. In this case we explicitly describe the left heart 
$$
\mathcal{LH}(\mathcal{A},\CC)\cong\hMork(\mathcal{A})[\mathcal{S}_{\CC}^{-1}]
$$
as the localization of $3$-term complexes $\ker f\rightarrow X\xrightarrow{\hspace{1pt}\scriptscriptstyle f\hspace{1pt}}Y$ with respect to an explicitly given multiplicative system $\mathcal{S}_{\CC}$. In our main Theorems \ref{NEW-LH} and \ref{KER-COK-PROP} we establish the multiplicative system axioms (MS1)\,--\,(MS3) directly, in particular cancellation and extension of fractions, and provide formulae for kernels and cokernels in $\mathcal{LH}(\mathcal{A},\CC)$\,---\,this makes $\mathcal{LH}(\mathcal{A},\CC)$ accessible to concrete computations. 
\smallskip

In the second part of the article we study concrete applications of the above results and their dual statements to the following categories appearing in the realm of functional analysis:\vspace{3pt}

\begin{myitemize}

\item[1.] $\COM$\,---\,the category of complete locally convex Hausdorff spaces was previously investigated by Prosmans \cite{Prosmans} who showed that it is not quasiabelian. We study two different notions of exactness: The maximal inflation exact structure has admissible cokernels and admits the right heart given by $\mathcal{RH}(\mathcal{A},\CC)\cong\hEpi(\mathcal{A})[\mathcal{S}_{\CC}^{-1}]$ as a surrogate. While this fact was mentioned already in \cite{HKRW}, we add an interpretation of the objects of the right heart as `formal co-quotients' and work out that the price to pay for the admissible cokernels is that deflations for the maximal inflation-exact structure need not to be surjective, see Remark \ref{COMRMK}. On the other hand, the maximal deflation exact structure on $\COM$ fails to have admissible kernels, but it is two-sided exact and consists precisely of the topologically exact sequences, see Theorem \ref{COM-THM}.

\vspace{3pt}

\item[2.] $\HCOM$\,---\,the category of hypo-complete locally convex Hausdorff spaces, where hypo-complete stands for either quasi-, sequentially or locally (a.k.a.\ Mackey) complete, appears naturally, e.g., in Lie theory, convenient calculus, semigroup theory or vector-valued holomorphy. All three categories fail to be extension closed in the category $\HDLCS$ of locally convex Hausdorff spaces endowed with the  topologically exact sequences \cite{DS12}. We show in Theorem \ref{HCOM-DEX} that nevertheless on all three categories the maximal deflation-exact structure consists precisely of the topologically exact sequences which yields an abelian surrogate as well as an exact embedding into $\HDLCS$.

\vspace{3pt}

\item[3.] $\FreDN$\,---\,the category of Fr\'echet spaces satisfying Vogt's \cite{MV} condition (DN) is a natural example of a category in which not all cokernels exist but which satisfies the assumptions of Theorem \ref{NEW-LH} when endowed with the maximal deflation-exact structure.

\vspace{3pt}

\item[4.] $\BMod(A)$\,---\,the category of complete convex bornological left modules over a unital bornological $\KK$-algebra $A$ is a quasiabelian category. We will however endow it with the linearly split exact structure as proposed by Meyer \cite{Meyer3} in the context of Hochschild cohomology. This example illustrates that our main results from Section \ref{SEC-LRH} accommodate a deliberately chosen non-maximal exact structure and provide an abelian surrogate accessible to explicit computations.

\end{myitemize} 

\smallskip

The Appendix gives a streamlined construction of the canonical left t-structure on the derived category of a strongly deflation-exact category with kernels. The argument uses a notion of acyclicity in a single degree adapted to this setting, whose relation to the convention used in \cite{HKRW} is also discussed.
\smallskip

A very prominent area that generates homological surrogates is tilting theory, see Sauter \cite[Chapter 7]{Sauter} for an overview focusing on exact categories. Let us mention in particular that Lawson, Letz, Sauter \cite{LLS26} recently gave results on surrogates, where two-sided exact categories were treated but the assumption of the existence of (co)kernels was relaxed, and that they indicated possible applications to complete LB-spaces.

\smallskip

For unexplained notation from category theory and functional analysis we refer to \cite{ML, Riehl} and \cite{MV, Jarchow, BPC, ProsmansSchneiders}, respectively. In particular, we refer to \cite{GZ, Krause, Mili,  Neeman, Verdier} for localization of (triangulated) categories. We fix $\KK$ to denote the field of real or complex numbers. We will use the word `map' synonymously for `morphism'; that includes cases where morphisms are linear and continuous or bounded maps. We will use the arrows $\rightarrowtail$ and $\twoheadrightarrow$  for in- and deflations while $\hookrightarrow$ and $\rightarrowtriangle$ indicate mono- and epimorphisms, respectively. We warn the reader that also the latter two notions depend on the category in question and that even if morphisms are some kind of maps, epimorphisms need not be surjective.

\smallskip

Since the categories of all topological or bornological vector spaces are not small, the applications presented in Sections \ref{SEC:CLEX}--\ref{SEC:BOR} require attention to size issues. While \cite{HKRW} assume essential smallness throughout, we will see that for our proofs in Section \ref{SEC-LRH} to hold it is sufficient that the bounded derived category of $\mathcal{A}$ is locally small. In several of our applications this is true due to some ad hoc argument. Otherwise one can achieve it by imposing a suitable size restriction. This will be explained in Remark \ref{SIZE}.

\section{Preliminaries}\label{SEC:PRELIM}

Let $\mathcal{A}$ be an additive category. We say that a morphism $f\in\mathcal{A}$ is a \emph{kernel} if there is a morphism $g\in\mathcal{A}$ with $f=\ker g$; \emph{cokernels} are defined dually. A pair $(f,g)$ of morphisms such that $f=\ker g$ and $g=\cok f$ will be referred to as a \emph{kernel-cokernel pair}. We say that $f\colon X\rightarrow Y$ is a \emph{semistable kernel} if it is a kernel and for any morphism $t\colon X\rightarrow T$ in $\mathcal{A}$ the pushout\vspace{-5pt}
 \begin{equation*}
\begin{tikzcd}
X\arrow{r}{f}\arrow{d}[swap]{t}\commutes[\mathrm{PO}]{dr} & Y \arrow{d}{q_Y}\\[4pt]
T \arrow{r}[swap]{q_T} & Q.
\end{tikzcd}
\end{equation*}
of $f$ along $t$ exists and $q_T$ is again a kernel. \emph{Semistable cokernels} are defined dually. A kernel-cokernel pair $(f,g)$ is a \emph{stable kernel-cokernel pair} if $f$ is a semistable kernel and $g$ is a semistable cokernel. If in $\mathcal{A}$ every morphism has a kernel [cokernel] then we say that $\mathcal{A}$ is a category \emph{with kernels} [\emph{with cokernels}]; some parts of the literature call such categories \emph{left} [\emph{right}] \emph{abelian}. A category with kernels and cokernels is \emph{preabelian}. If additionally all kernels and cokernels are semistable, then $\mathcal{A}$ is called \emph{quasiabelian}.

\smallskip

In a quasiabelian category $\mathcal{A}$, homological algebra is possible\,---\,subject to a few caveats\,---\,by treating the class of all kernel-cokernel pairs as short exact sequences. In what follows, we will however focus on categories that are either \emph{not} quasiabelian, or are quasiabelian but we deliberately choose \emph{not} to declare every kernel-cokernel pair as a short exact sequence.

\begin{dfn}\label{DEFLEXCAT}\cite[Dfns 2.7 and 4.6]{HKRW}\,(and \cite{HKR, HR20}) We call a pair $(\mathcal{A},\CC)$, consisting of an additive category $\mathcal{A}$ and a class $\CC$ of kernel-cokernel pairs in $\mathcal{A}$ that is closed under isomorphisms, a \emph{conflation category} . If $(f,g)\in\CC$, then we refer to $f$ as an \emph{inflation} and to $g$ as a \emph{deflation}. A conflation category $(\mathcal{A},\CC)$ \vspace{3pt}
\begin{compactitem}
\item[(i)] is \emph{deflation-exact} if the following three axioms hold.

\vspace{2pt}

\begin{compactitem}
\item[\phantom{x}\color{RoyalBlue}{\textbf{R0}}\label{R0}] For each $X\in\mathcal{A}$ the map $X\rightarrow0$ is a deflation.\vspace{1.5pt}
\item[\phantom{x}\color{RoyalBlue}{\textbf{R1}}\label{R1}] The composition of two deflations is a deflation.\vspace{1.5pt}
\item[\phantom{x}\color{RoyalBlue}{\textbf{R2}}\label{R2}] The pullback of a deflation along any morphism exists and is again a deflation.
\end{compactitem}

\vspace{4pt}

\item[(ii)] is \emph{inflation-exact} if the following three axioms hold.

\vspace{2pt}

\begin{compactitem}
\item[\phantom{x}\color{RoyalBlue}{\textbf{L0}}\label{L0}] For each $X\in\mathcal{A}$ the map $0\rightarrow X$ is an inflation.\vspace{1.5pt}
\item[\phantom{x}\color{RoyalBlue}{\textbf{L1}}\label{L1}] The composition of two inflations is an inflation.\vspace{1.5pt}
\item[\phantom{x}\color{RoyalBlue}{\textbf{L2}}\label{L2}] The pushout of an inflation along any morphism exists and is again an inflation.
\end{compactitem}

\vspace{4pt}

\item[(iii)] is \emph{strongly deflation-exact} [\emph{strongly inflation-exact}] if additionally the corresponding one of the following two axioms holds.

\vspace{3pt}

\begin{compactitem}
\item[\phantom{x}\color{RoyalBlue}{\textbf{R3}}\label{R3}] If $p$ has a kernel and $p\hspace{1pt}i$ is a deflation, then $p$ is a deflation.
\vspace{1.5pt}
\item[\phantom{x}\color{RoyalBlue}{\textbf{L3}}\label{L3}] If $i$ has a cokernel and $p\hspace{1pt}i$ is an inflation, then $i$ is an inflation.
\end{compactitem}

\vspace{4pt}

\item[(iv)] has \emph{admissible kernels} [\emph{admissible cokernels}] if in $\mathcal{A}$ every morphism possesses a kernel [cokernel] and if every kernel [cokernel] is an inflation [a deflation].

\vspace{4pt}

\item[(v)] is \emph{exact} if \hypref{R0}\hspace{2pt}--\hspace{1pt}\hypref{R3} as well as \textcolor{RoyalBlue}{\textbf{L0}}\hspace{1pt}--\hspace{1pt}\textcolor{RoyalBlue}{\textbf{L3}} all hold.\diam{}
\end{compactitem}
\end{dfn}

Note that other authors discussed the same concepts as above with some subtle differences in notation: In \cite{HR19,HR19a,HR24} the axioms \hypref{R1}\hspace{1pt}--\hspace{1pt}\hypref{R3} are defined as above, but our axiom \hypref{R0} is there called \textbf{R0}$^{\boldsymbol{\ast}}$. The definition with \hypref{R0} as above guarantees that all split kernel-cokernel pairs are conflations, cf.~\cite[Prop 2.6]{HR24} and \cite[Rmk 2.9]{HKRW}. In \cite{HR19} the term `right exact' is used synonymously for `deflation-exact', while in \cite{BC13} deflation-exact categories are called `left exact'; their axioms (R0$^*$)\,--\,(R3) correspond to our axioms \hypref{L0}\hspace{1pt}--\hspace{1pt}\hypref{L3}. Rump \cite{RumpMax} also uses the term `left exact' but includes \hypref{R3} in the definition. It is well-known \cite{Keller90} that the axioms \hypref{R0}\hspace{2pt}--\hspace{1pt}\hypref{R2} and \hypref{L2}, or dually \hypref{L0}\hspace{2pt}--\hspace{1pt}\hypref{L2} and \hypref{R2}, suffice to define an exact category in Quillen's sense \cite{Quillen72}.

\begin{rmk}\label{RMK-1}\cite[Rmk 4.9, Prop 4.12 + dual statement]{HKRW} A deflation-exact [inflation-exact] category with admissible kernels [cokernels] is automatically strongly deflation [inflation] exact. In any such category all kernels are inflations and all cokernels are deflations; the conflations are given by all kernel-cokernel pairs.\hfill\diam{}
\end{rmk}

The next result identifies natural (maximal) conflation structures under assumptions weaker than $\mathcal{A}$ being quasiabelian. Let us from now on denote by $\CC_{\text{all},\,\mathcal{A}}$ the class of all kernel-cokernel pairs on any given additive category $\mathcal{A}$. Recall further that $\mathcal{A}$ is \emph{karoubian} if every idempotent in $\mathcal{A}$ has a kernel, or, equivalently, a cokernel. 

\begin{thm}\label{THMDEFMAX}{\rm(}\cite[Thm 3.5]{Crivei}, \cite[Prop 9.2]{HR24}, \cite[Thm 2.5]{Wegner25}{\rm)} Let $\mathcal{A}$ be karoubian.\vspace{2pt}

\begin{compactitem}

\item[(i)] $\EE_{\operatorname{max},\,\mathcal{A}}:=\bigl\{(f,g)\in\CC_{\operatorname{all},\,\mathcal{A}}\:|\:(f,g) \text{ is a stable kernel-cokernel pair}\hspace{0.5pt}\bigr\}$ is the maximal exact structure on $\mathcal{A}$.

\vspace{3pt}

\item[(ii)]  $\DD_{\operatorname{max},\,\mathcal{A}}:=\bigl\{(f,g)\in\CC_{\operatorname{all},\,\mathcal{A}}\:|\:g \text{ is semistable cokernel}\hspace{0.5pt}\bigr\}$ is a strongly deflation-exact structure and the maximal deflation exact structure on $\mathcal{A}$.

\vspace{3pt}

\item[(iii)]  $\II_{\operatorname{max},\,\mathcal{A}}:=\bigl\{(f,g)\in\CC_{\operatorname{all},\,\mathcal{A}}\:|\:f \text{ is semistable kernel}\hspace{0.5pt}\bigr\}$ is a strongly inflation-exact structure and the maximal inflation exact structure on $\mathcal{A}$.\diam{}
\end{compactitem}
\end{thm}

We need to make some remarks on size.

\begin{rmk}\label{SIZE}As our methods below employ the derived category, and in parts the Gabriel-Quillen embedding, while our applications involve categories of topological or bornological vector spaces with arbitrary dimension, set-theoretic problems might occur. In that respect we first point out that for the approach given in Section \ref{SEC-LRH} it indeed is sufficient to assume that $\Der^{\bb}(\mathcal{A},\CC)$ is locally small; this is what we will assume implicitly for the whole paper, unless stated otherwise. In fact, it is known to hold for the classical categories $\LCS$, $\Ban$, $\Fre$ and $\Hilb$, since $\LCS$, $\Ban$ and $\Fre$ have enough injective objects, see, e.g., \cite[Prop 2.1.12(iii), beginning of Section 4.3, Prop 4.4.6]{Prosmans}, and since in $\Hilb$ every short exact sequence splits. It also holds for $\HDLCS$ using $\Der^{\bb}(\HDLCS)\cong\Der^{\bb}(\LCS)$, which follows analogously to \cite[Prop 3.2.7]{Prosmans} and for $\Bor$, $\SBor$ and $\CBor$ since these categories have enough projectives \cite[Props 2.13, 4.11, 5.8]{ProsmansSchneiders}. More examples will be given in due course. If no `ad hoc' argument as before is available, or if we want to apply results that require (essential) smallness of $\mathcal{A}$, see \cite[p.~442]{HKRW}, we will (implicitly) consider restricted versions of our topological or bornological categories. More precisely, we then follow \cite[Section 8]{Shulman}, fix Grothendieck universes $\mathcal{U}\in\mathcal{V}$ such that $\KK$ belongs to $\mathcal{U}$ and take for $\mathcal{A}$, e.g., all topological vector spaces $(X,\tau)$ enjoying a certain property \emph{and} such that the underlying set $X$ belongs to $\mathcal{U}$. In this way $\mathcal{A}$ is $\mathcal{V}$-small. Writing $\mathcal{U}=V_{\kappa}$ in the notation of \cite[Section 4]{Shulman} with $\kappa$ an inaccessible cardinal, we may alternatively define $\mathcal{A}$ to consist of those $(X,\tau)$ such that $|X|<\kappa$ holds. This category will still be \emph{essentially} small with respect to $\mathcal{V}$. We point out that with both approaches one has to devote some attention to notions that quantify over the whole category (like kernels, cokernels, completions, etc.) and investigate how these are affected by the size restriction. Since in all our applications any construction of that sort turns out to be independent of the restriction, we will in due course mostly omit discussions about size and only refer to the aforementioned explanations as well as to \cite{Shulman}.
\end{rmk}

Next we recall the definition of the derived category. Since this simplifies the definition and is the case for all concrete categories that we consider in the remaining sections, we restrict ourselves to the case that $(\mathcal{A},\CC)$ is (at least) strongly inflation- or strongly deflation-exact, and $\mathcal{A}$ is karoubian. We then denote by $\C(\mathcal{A})$ the category of cochain complexes and by $\K(\mathcal{A})$ the homotopy category furnished with the triangulated structure induced by the mapping cone triangles where we, given $f\colon X^{\bullet}\rightarrow Y^{\bullet}$, use the convention
$$
(\operatorname{cone}f)^n=X^{n+1}\oplus Y^{n} \;\text{ and }\; d_{\operatorname{cone}f}^{\hspace{1pt}n}=\begin{bmatrix}-d_{X}^{\hspace{1pt}n+1} & 0\\\phantom{-}f^{n+1}&d_Y^{\hspace{1pt}n}\end{bmatrix}
$$
for the cone and $\Sigma\colon\K(\mathcal{A})\rightarrow\K(\mathcal{A})$, $(\Sigma{}X)^n=X^{n+1}$, $d_{\Sigma{}X}^{\hspace{1pt}n}=-d_{X}^{\hspace{1pt}n+1}$ for the suspension. For $*\in\{+,-,\text{b}\}$ let $\K^{\ast}(\mathcal{A})$ be the full triangulated subcategory of $\K(\mathcal{A})$ formed by complexes that are homotopic to left bounded, right bounded and bounded complexes, respectively.

\smallskip

\begin{dfn}\label{ACYCLIC}\cite[Dfn 2.14, text after Lem 2.15 and Prop 2.17(3)]{HKRW} Let $(\mathcal{A},\CC)$ be strongly deflation- or inflation-exact and karoubian.\vspace{3pt}
\begin{compactitem}
\item[(i)] A cochain complex $C^{\bullet}=(C^n,d^{\hspace{1pt}n})_{n\in\ZZ}$ is \emph{$\CC$-acyclic} if there exist factorizations of all differentials as follows with $(i^{n-2},p^{n-1})\in\CC$ for all $n\in\ZZ$:\vspace{-3pt}
\begin{equation*}
\begin{tikzcd}[column sep = 1.5em, row sep = 1.2em]
& \cdots \ar[rr] 
&& C^{\hspace{0.5pt}n-1} \ar[rr,"d^{\hspace{0.5pt}n-1}"] \ar[dr,"p^{n-1}"']
&& C^{\hspace{0.5pt}n}     \ar[rr,"d^{\hspace{0.5pt}n}"] \ar[dr,"p^{n}"']
&& C^{\hspace{0.5pt}n+1} \ar[rr] \ar[drr,"p^{n+1}"']
&& \cdots \\
&\phantom{\cdots}\begin{picture}(0,0)(0,0)\put(-12,6){$\cdots$}\end{picture}\ar[urr,"i^{n-2}"']
&&
&\phantom{!}Z^n\phantom{!}\ar[ur,"i^{n-1}"']
&
&Z^{n+1}\ar[ur,"i^{n}"']
&
&&\phantom{\cdots}\phantom{\cdots}\begin{picture}(0,0)(0,0)\put(-18,6){$\cdots$}\end{picture}
&
\end{tikzcd}
\end{equation*}

\item[(ii)] By $\Ac(\mathcal{A},\CC)$ we denote the $\CC$-acyclic complexes. As a subcategory $\Ac(\mathcal{A},\CC)\subseteq\K(\mathcal{A})$ it is thick, triangulated and in particular replete. For $*\in\{+,-,\text{b}\}$ we put $\Ac^{\ast}(\mathcal{A},\CC):=\Ac(\mathcal{A},\CC)\cap\K^{\ast}(\mathcal{A})$ which is a thick subcategory of $\K^{\ast}(\mathcal{A})$ as well.

\vspace{3pt}

\item[(iii)] Let $*\in\{\emptyset,+,-,\text{b}\}$. We define the derived category of $(\mathcal{A},\CC)$ as the Verdier quotient
\[
\Der^{\ast}(\mathcal{A},\CC)=\,^{\textstyle\K^{\ast}(\mathcal{A})}\big/_{\textstyle\Ac^{\ast}(\mathcal{A},\CC)}=\K^{\ast}(\mathcal{A})\hspace{0.5pt}[\hspace{1.5pt}\mathcal{N}^{\hspace{1pt}*}(\mathcal{A},\CC)^{-1}],
\]
where $\mathcal{N}^{\hspace{1pt}*}(\mathcal{A},\CC)=\bigl\{f^{\bullet}\colon X^{\bullet}\rightarrow Y^{\bullet}\:\big|\:\cone(f^{\bullet})\in\Ac^{\ast}(\mathcal{A},\CC)\bigr\}$ is a saturated multiplicative system of quasi-isomorphisms. We furnish $\Der^{\ast}(\mathcal{A},\CC)$ with the natural triangulated structure.\diam{}

\end{compactitem}
\end{dfn}

Notice that in Definition \ref{ACYCLIC}(i) automatically $i^{n-1}=\ker d^{\hspace{1pt}n}$ and $p^{n-1}=\cok d^{\hspace{1pt}n-2}$ hold for all $n\in\ZZ$. This was required in \cite[Dfn 2.14]{HKRW} explicitly, compare Definition \ref{DFN-AC-N}, the remarks after it and Appendix \ref{SEC-APP}.

\smallskip

Let us now summarize some properties of the derived category that are standard in the abelian case and remain valid in our setting.

\begin{thm}\label{RMK-DC}\cite[Prop 3.20, Prop 6.2, Rmk 3.5(2) and Prop 6.3 and dual statements]{HR19} Let $(\mathcal{A},\CC)$ be a strongly deflation- or inflation-exact karoubian category and $*\in\{+,-,\bb,\emptyset\}$.\vspace{3pt}
\begin{compactitem}

\item[(i)] The natural functor $i\colon\mathcal{A}\rightarrow\emph{\Der}^{\ast}(\mathcal{A},\CC)$, that maps $X\in\mathcal{A}$ to the stalk complex $\cdots\rightarrow0\rightarrow X\rightarrow0\rightarrow\cdots$ with $X$ in degree zero, is fully faithful.

\vspace{3pt}

\item[(ii)] A sequence $X\rightarrow Y\rightarrow Z$ in $\mathcal{A}$ is a conflation iff there is a distinguished triangle  $i(X)\rightarrow i(Y)\rightarrow i(Z)\rightarrow i(X)[1]$  in $\emph{\Der}^{\ast}(\mathcal{A},\CC)$.

\vspace{3pt}

\item[(iii)] A sequence $X\rightarrow Y\rightarrow Z$ in $\mathcal{A}$ is a conflation iff the complex $\cdots\rightarrow 0\rightarrow X\rightarrow Y\rightarrow Z\rightarrow 0\rightarrow\cdots$ is $\CC$-acyclic.

\vspace{3pt}

\item[(iv)] If $(\mathcal{A},\CC)$ is exact, then $i(\mathcal{A})\subseteq\emph{\Der}^{\ast}(\mathcal{A},\CC)$ is extension closed.\diam{}
\end{compactitem}
\end{thm}







\section{Left and Right Hearts}\label{SEC-LRH}

Let $(\mathcal{A},\CC)$ be a strongly deflation-exact category with kernels. In \cite[Section 3]{HKRW} the authors defined for a complex $C^{\bullet}$ over $\mathcal{A}$ and $n\in\ZZ$ the canonical `left' truncations
\begin{equation}\label{TRUNC}
\begin{tikzcd}[column sep=17pt,
row sep=1pt,
column 1/.style={
    nodes={
      text width=4cm,
      align=right
    }
  }]
\mathmakebox[2cm][r]{\tau_{\LL}^{\leqslant n}C^{\bullet}}
  &[-20pt]=
  &[-20pt] \cdots\arrow[r]
  & C^{\hspace{0.5pt}n-2} \arrow[r]
  & C^{\hspace{0.5pt}n-1} \arrow[r, "p^{n-1}"]
  & \ker d^{\hspace{1pt}n} \arrow[r]
  & 0 \arrow[r]
  & \cdots
\\
\mathmakebox[2cm][r]{\tau_{\LL}^{\geqslant n+1}C^{\bullet}}
  &[-20pt]=
  &[-20pt] \cdots \arrow[r]
  & 0 \arrow[r]
  & \ker d^{\hspace{1pt}n} \arrow[r, "i^{n-1}"]
  & C^{\hspace{0.5pt}n} \arrow[r]
  & C^{\hspace{0.5pt}n+1} \arrow[r]
  & \cdots
\end{tikzcd}
\end{equation}
where we added an `L' to the notation to distinguish from the dual version to be discussed later. There are natural maps $\tau_{\LL}^{\leqslant n}C^{\bullet}\rightarrow C^{\bullet}$ and $C^{\bullet}\rightarrow\tau_{\LL}^{\geqslant n+1}C^{\bullet}$. The truncation functors lead to the left t-structure $(\Der^{\leqslant0}_{\LL}(\mathcal{A},\CC),\Der^{\geqslant 0}_{\LL}(\mathcal{A},\CC))$ on the unbounded derived category of $(\mathcal{A},\CC)$ given by
\begin{equation}\label{LEFT-T-S}
\begin{aligned}
\Der^{\leqslant0}_{\LL}(\mathcal{A},\CC) & =  \bigl\{C^{\bullet}\in\Der(\mathcal{A},\CC)\:|\:\tau_{\LL}^{\geqslant1}C^{\bullet}\cong0\bigr\}\\[-3pt]
& =  \bigl\{C^{\bullet}\in\Der(\mathcal{A},\CC)\:|\:\forall\:n\geqslant 1\colon\ker d^{\hspace{1pt}n-1}\stackrel{i^{n-2}}{\longrightarrow}C^{n-1}\stackrel{p^{n-1}}{\longrightarrow}\ker d^{\hspace{1pt}n}\in\CC\bigr\},\\[5pt]
\Der^{\geqslant0}_{\LL}(\mathcal{A},\CC) & =  \bigl\{C^{\bullet}\in\Der(\mathcal{A},\CC)\:|\:\tau_{\LL}^{\leqslant-1}C^{\bullet}\cong0\bigr\}\\[-3pt]
& =  \bigl\{C^{\bullet}\in\Der(\mathcal{A},\CC)\:|\:\forall\:n\leqslant-1\colon\ker d^{\hspace{1pt}n-1}\stackrel{i^{n-2}}{\longrightarrow}C^{n-1}\stackrel{p^{n-1}}{\longrightarrow}\ker d^{\hspace{1pt}n}\in\CC\bigr\}.
\end{aligned}
\end{equation}
We denote by $\mathcal{LH}(\mathcal{A},\CC) = \Der^{\leqslant0}_{\LL}(\mathcal{A},\CC)\cap \Der^{\geqslant0}_{\LL}(\mathcal{A},\CC)\subseteq\Der(\mathcal{A},\CC)$ the so-called \emph{left heart} of the left t-structure and obtain the left cohomology functors
\begin{equation}\label{EQ-LH}
\begin{array}{rcl}
\operatorname{LH}^{\hspace{0.5pt}n}&\hspace{-4pt}=&\hspace{-4pt}\tau_{\LL}^{\leqslant0}\circ\tau_{\LL}^{\geqslant0}\circ\Sigma^n\colon \Der(\mathcal{A},\CC)\rightarrow\mathcal{LH}(\mathcal{A},\CC)\\[4pt]
\operatorname{LH}^{\hspace{0.5pt}n}(C^{\bullet})&\hspace{-4pt}=&\hspace{-4pt}\cdots\rightarrow0\rightarrow\ker d^{\hspace{1pt}n-1}\xrightarrow{\pm i^{\hspace{0.5pt}n-2}}C^{n-1}\xrightarrow{\pm p^{n-1}}\ker d^{\hspace{1pt}n}\rightarrow0\rightarrow\cdots\end{array}
\end{equation}
where $\ker d^{\hspace{1pt}n}$ is in degree $0$ and $\operatorname{LH}^{\hspace{0.5pt}n}(f^{\bullet}\colon X^{\bullet}\rightarrow Y^{\bullet})$ has as components $f^{n-1}$ in degree $-1$ and the natural maps between the kernels. The embedding $i\colon\mathcal{A}\rightarrow\Der(\mathcal{A},\CC)$ corestricts to an embedding $\phi\colon\mathcal{A}\rightarrow\mathcal{LH}(\mathcal{A},\CC)$ that sends objects to stalks. Properties of $\mathcal{LH}(\mathcal{A},\CC)$ and $\phi$ are summarized in the following theorem.

\begin{thm}\label{THM-LH-PHI}\cite[Prop 3.9, Prop 3.11, Thm 3.12]{HKRW} Let $(\mathcal{A},\CC)$ be a strongly deflation-exact category with kernels. Then $\mathcal{LH}(\mathcal{A},\CC)$ is an abelian category. Moreover:\vspace{3pt}
\begin{compactitem}

\item[(i)] The embedding $\phi\colon(\mathcal{A},\CC)\rightarrow\mathcal{LH}(\mathcal{A},\CC)$ is fully faithful, exact, reflects exactness, commutes with kernels and preserves and reflects monomorphisms.

\vspace{3pt}

\item[(ii)] It lifts to a triangle equivalence $\emph{\Der}^*(\mathcal{A},\CC)\rightarrow\emph{\Der}^*(\mathcal{LH}(\mathcal{A},\CC))$ for $*\in\{-,\emph{\bb},\emptyset\}$.\hfill\diam{}

\end{compactitem}
\end{thm}

We now give an explicit description of the left heart without assuming that all kernels are admissible. The proof verifies the multiplicative system axioms constructively; for brevity, the corresponding formulas are omitted from the statement of the theorem.

\begin{thm}\label{NEW-LH} Let $(\mathcal{A},\CC)$ be a strongly deflation-exact category with kernels. Then we have the following description of the left heart $\mathcal{LH}(\mathcal{A},\CC)\cong{}\hMork(\mathcal{A})[\mathcal{S}^{-1}_{\CC}]$:\vspace{2pt}
\begin{myitemize}

\item[(i)] The objects of $\hMork(\mathcal{A})$ are sequences of the form $\ker f\hookrightarrow X'\stackrel{\scriptscriptstyle f}{\rightarrow}X$ in $\mathcal{A}$ which we will abbreviate by $X_f$.

\item[(ii)]  A morphism $\alpha=(\alpha'',\alpha',\alpha) \colon X_f\rightarrow Y_g$ in $\hMork(\mathcal{A})$ is a commutative diagram
\begin{equation*}
\begin{tikzcd}
\ker f\ar[r,hook,"k_f"]\ar[d,swap,"\alpha''"]& X'\arrow{r}{f}\arrow{d}[swap]{\alpha'} & X\ar[d,"\alpha"] \\[4pt]
\ker g\ar[r,hook,swap,"k_g"]& Y' \arrow{r}[swap]{g} & Y
\end{tikzcd}
\end{equation*}
modulo the equivalence relation $\alpha\sim\beta$ iff $\exists\:\rho\colon X\rightarrow Y'\colon\alpha-\beta=g\rho$.\vspace{3pt}

\item[(iii)] The class $\mathcal{S}_{\CC}$ consists of all morphisms $\alpha$ for which $\cone(\alpha)$ factorizes as
\begin{equation*}
\begin{tikzcd}[column sep = 1.5em, row sep = 1.2em, ampersand replacement=\&]
 \cone(\alpha):=\ker f \ar[rr,->,"{\bigl[\begin{smallmatrix}-k_f\\\phantom{-}\alpha''\end{smallmatrix}\bigr]}"] 
\&\& X'\oplus\ker g\ar[rr,"{\bigl[\begin{smallmatrix}-f&0\\\phantom{-}\alpha'&k_g\end{smallmatrix}\bigr]}"] \ar[dr,->,"r"']
\&\& X\oplus Y' \ar[rr,->,"{[\hspace{1pt}\alpha\;\,g\hspace{1pt}]}"]
\&\& Y\\
\&\&\&\ker{[\hspace{1pt}\alpha\;\,g\hspace{1pt}]}\ar[ur,->,"{k_{[\alpha\,g]}}"']\&\&\&\&
\end{tikzcd}
\end{equation*}
with $(\bigl[\begin{smallmatrix}-k_f\\\phantom{-}\alpha''\end{smallmatrix}\bigr],r)$ and $(k_{[\alpha\,g]},[\hspace{1pt}\alpha\;\,g\hspace{1pt}])$ belonging to  $\CC$.\vspace{3pt}

\item[(iv)] The class $\mathcal{S}_{\CC}$ is a saturated multiplicative system, i.e., the morphisms of $\hMork(\mathcal{A})[\mathcal{S}^{-1}_{\CC}]$ can be represented as left fractions $\sigma^{-1}\alpha\equiv(X_f\stackrel{\hspace{-3.5pt}\alpha}{\longrightarrow}Z_h\stackrel{\sigma\hspace{-2pt}}{\longleftarrow}Y_g)$ or right fractions $\alpha\hspace{0.5pt}\sigma^{-1}\equiv(X_f\stackrel{\sigma\hspace{-2pt}}{\longleftarrow}Z_h\stackrel{\hspace{-3.5pt}\alpha}{\longrightarrow}Y_g)$ with $\sigma\in S_{\CC}$ and there is a calculus to compute with them. 
\end{myitemize}

\end{thm}
\begin{proof} We read (i) as the definition of the objects of $\hMork(\mathcal{A})$ and observe quickly that the equivalence relation stated in (ii) coincides with usual homotopy if we read the objects as complexes: Indeed, if a chain map $\alpha$ is null-homotopic, then a map $\rho$ with the given property exists by definition. For the other direction we observe that $\alpha=g\rho$ implies $g(\alpha'-\rho f)=0$ from whence we get from the universal property of the kernel of $g$ a map $\rho'\colon X'\rightarrow\ker g$ with $\alpha'=k_g\rho'+\rho f$. Since $k_g$ is monic it then follows $\alpha''=\rho'k_f$. Part (iii) completes the definition of the morphisms of $\hMork(\mathcal{A})[\mathcal{S}_{\CC}^{-1}]$; we observe that $\cone(\alpha)$ from (iii) as a complex is the mapping cone of $\alpha$ if we read the latter as a chain map. 

\smallskip

Below we will firstly establish that $\mathcal{S}_{\CC}$ is a multiplicative system before returning to the claimed equivalence $\mathcal{LH}(\mathcal{A},\CC)\cong\hMork(\mathcal{A})[\mathcal{S}_{\CC}^{-1}]$.

\smallskip

We verify the conditions (MS1)\hspace{1pt}--\hspace{1pt}(MS3) in the notation of \cite[Section 1.3]{KrauseChicago}. Indeed, (MS1) follows as $\Ac(\mathcal{A},\CC)\subseteq\K(\mathcal{A},\CC)$ is a triangulated subcategory and thus $\mathcal{N}(\mathcal{A},\CC)$ is a multiplicative system by what we noted in Definition \ref{ACYCLIC}(ii)\hspace{1pt}--\hspace{1pt}(iii) and in the first paragraph of this proof. The two remaining conditions we are going to prove explicitly, cf.~Remark \ref{MS-RMK}.
\smallskip

\textcircled{1} We begin with the first part of (MS2). Let $\alpha\colon X_f\rightarrow Z_h$ and $\tau\colon Y_g\rightarrow Z_h$ be morphisms with $\tau\in\mathcal{S}_{\CC}$, i.e.,
\[
\begin{tikzcd}
\ker f\ar[r,hook,"k_f"]\ar[d,swap,"\alpha''"]& X'\arrow{r}{f}\arrow{d}[swap]{\alpha'} & X\ar[d,"\alpha"] \\[4pt]
\ker h\ar[r,hook,"k_h"]& Z' \arrow{r}{h} & Z\\[4pt]
\ker g\ar[r,hook,swap,"k_g"]\ar[u,"{\tau''}"]& Y'\ar[u,"{\tau'}"] \arrow{r}[swap]{g} & Y\ar[u,swap,"{\tau}"]
\end{tikzcd}
\]
commutes. We use the pullback $S$ of $h$ along $[\alpha\;-\hspace{-1pt}\tau]$ and $S':=X'\oplus Y'\oplus\ker h$ to get
\begin{equation}\label{EQ-PB-PRF}
\begin{tikzcd}[ampersand replacement=\&]
S'\ar[rd,"u"]\ar[rrd,bend left =25, "{\bigl[\begin{smallmatrix}f&0&0\\0&g&0\end{smallmatrix}\bigr]}"]\ar[rdd, bend right=35,swap, "{[\alpha'\;-\tau'\;\,k_h]}"]\& \& \\
\&S\ar[d,swap,"{s_2}"]\ar[r,"{s_1}"] \commutes[\mathrm{PB}]{dr} \& X\oplus Y\ar[d, "{[\alpha\;-\tau]}"]\\
\& Z'\ar[r,swap,"{h}"] \& Z
\end{tikzcd}
\end{equation}
and claim that
\[
k_u:=\Bigl[\begin{smallmatrix}k_f & 0\\0&k_g\\-\alpha''&\tau''\end{smallmatrix}\Bigr]\colon \ker u:=\ker f\oplus\ker g\rightarrow S'
\]
is a kernel of $u$. Indeed, for $i=1,2$ we calculate $s_iu\hspace{1pt}k_u=0$ and then use the pullback diagram above to conclude $u\hspace{1pt}k_u=0$. Let then $t:=[t_1\;t_2\;t_3]^{\T}\colon T\rightarrow S'$ be given with $u\hspace{1pt}t=0$. In particular $\bigl[\begin{smallmatrix}ft_1\\gt_2\end{smallmatrix}\bigr]=s_1u\hspace{1pt}t=\bigl[\begin{smallmatrix}0\\0\end{smallmatrix}\bigr]$ and we get unique $a_1\colon T\rightarrow \ker f$, $a_2\colon T\rightarrow\ker g$ with $t_1=k_fa_1$, $t_2=k_ga_2$. Using that $[\hspace{1pt}\alpha'\;-\hspace{-2pt}\tau'\;k_h\hspace{1pt}]\hspace{1pt}t=0$ and that $k_h$ is monic we obtain $k_u\bigl[\begin{smallmatrix}a_1\\a_2\end{smallmatrix}\bigr]=t$ and we see that $\bigl[\begin{smallmatrix}a_1\\a_2\end{smallmatrix}\bigr]$ is unique with this property. We thus defined an object $S_u\in\hMork(\mathcal{A})$ and continue with defining maps $\sigma\colon S_u\rightarrow X_f$ and  $\beta\colon S_u\rightarrow Y_g$, i.e.,
\[
\begin{tikzcd}
\ker f\ar[r,hook,"k_f"]& X'\arrow{r}{f} & X\\[4pt]
\ker u\ar[r,hook,"k_u"]\ar[d,swap,"{\beta''}"]\ar[u,"\sigma''"]& S' \arrow{r}{u}\ar[d,swap,"{\beta'}"]\arrow{u}{\sigma'} & S\ar[u,swap,"\sigma"]\ar[d,"{\beta}"] \\[4pt]
\ker g\ar[r,hook,swap,"k_g"]& Y' \arrow{r}[swap]{g} & Y
\end{tikzcd}
\]
by letting $\sigma'',\sigma',\beta'',\beta'$ be the natural projections and $\sigma:=[\hspace{1pt}1\;\,0\hspace{1pt}]s_1$, $\beta:=[\hspace{1pt}0\;\,1\hspace{1pt}]s_1$. To see that the square
\begin{equation*}
\begin{tikzcd}
S_u\arrow{r}{\beta}\arrow{d}[swap]{\sigma} & Y_g \arrow{d}{\tau}\\[4pt]
X_f \arrow{r}[swap]{\alpha} & Z_h
\end{tikzcd}
\end{equation*}
is commutative in $\hMork(\mathcal{A})$ we consider $\alpha\sigma-\tau\beta\colon S_u\rightarrow Z_h$ and use (ii) with $\rho=s_2$ to conclude that the former is null-homotopic. With $u=[u_1\;u_2\;u_3]$ we compute $\cone(\sigma)$:
\begin{equation*}
\ker f\oplus\ker g\xrightarrow{\scalebox{0.85}{$\left[
\begin{smallmatrix}
-k_f&0\\
0 & -k_g\\
\alpha''&-\tau''\\
1 & 0
\end{smallmatrix}
\right]$}}X'\oplus Y'\oplus\ker h\oplus\ker f\xrightarrow{\scalebox{0.85}{$\bigl[
\begin{smallmatrix}
-u_1&-u_2&-u_3&0\\
1&0&0&k_f
\end{smallmatrix}
\bigr]$}}S\oplus X'\xrightarrow{\scalebox{0.85}{$\scalebox{1.005}{$[$}\hspace{1pt}\raisebox{0.75pt}{$\scriptstyle[\hspace{1pt}1\;0\hspace{1pt}]s_1\;\,f\hspace{1pt}$}\scalebox{1.005}{$]$}$}} X.
\end{equation*}
We change the order in the second entry and subtract the two null-homotopic direct summands $\cdots\rightarrow0\rightarrow\ker f\stackrel{\hspace{-3pt}\scriptscriptstyle1}{\rightarrow}\ker f\rightarrow0\rightarrow\cdots$ and $\cdots\rightarrow0\rightarrow X'\stackrel{\hspace{-3pt}\scriptscriptstyle1}{\rightarrow}X'\rightarrow0\rightarrow\cdots$. Establishing acyclicity of $\cone(\sigma)$ is then reduced to considering the sequence
\begin{equation}\label{EQ-Q}
\ker g\xrightarrow{\scalebox{0.95}{$\bigl[
\begin{smallmatrix}
-k_g \\
\tau''
\end{smallmatrix}
\bigr]$}}Y'\oplus\ker h\xrightarrow{\scalebox{0.80}{$[\hspace{1pt}-u_2\;\,u_3\hspace{1pt}]$}}S\xrightarrow{\scalebox{0.80}{$[\hspace{1pt}1\;\,0\hspace{1pt}]s_1$}} X.
\end{equation}
Since $\cone(\tau)$ is acyclic we have the following factorization
\begin{equation}\label{EQ-CONE-TAU}
\begin{tikzcd}[column sep = 1.5em, row sep = 1.2em, ampersand replacement=\&]
 \ker g \ar[rr,>->,"{\bigl[\begin{smallmatrix}-k_g\\\phantom{-}\tau''\end{smallmatrix}\bigr]}"] 
\&\& Y'\oplus\ker h     \ar[rr,"{\bigl[\begin{smallmatrix}-g&0\\\phantom{-}\tau'&k_h\end{smallmatrix}\bigr]}"] \ar[dr,->>,"r"']
\&\& Y\oplus Z' \ar[rr,->>,"{[\hspace{1pt}\tau\;\,h\hspace{1pt}]}"]
\&\& Z. \\
\&\&\&\ker{[\hspace{1pt}\tau\;\,h\hspace{1pt}]}\ar[ur,>->,"{k_{[\tau\,h]}}"']\&\&\&\&
\end{tikzcd}
\end{equation}
Using \eqref{EQ-PB-PRF} we first see that the following diagram is a pullback and employing \eqref{EQ-CONE-TAU} it follows that the map at the top is a deflation:
\begin{equation}\label{EQ-ZETA}
\begin{tikzcd}[column sep = 2.6em, ampersand replacement=\&]
S\ar[r,->>,"{[\hspace{1pt}1\;\,0\hspace{1pt}]s_1}"]\arrow{d}[swap]{\bigl[\begin{smallmatrix}[\hspace{1pt}0\;\,1\hspace{1pt}]s_1\\s_2\end{smallmatrix}\bigr]} \& X \arrow{d}{\alpha}\\[4pt]
Y\oplus Z' \ar[r,->>,swap, "{\hspace{-5pt}[\hspace{1pt}\tau\;\,h\hspace{1pt}]}"] \& Z.
\end{tikzcd}
\end{equation}
As parallel maps in a pullback have isomorphic kernels, we write $\zeta\colon\ker[\hspace{1pt}\tau\;\,h\hspace{1pt}]\rightarrow S$ for the kernel of $[\hspace{1pt}0\;\,1\hspace{1pt}]s_1$ and note that $\bigl[\begin{smallmatrix}[\hspace{1pt}1\;\,0\hspace{1pt}]s_1\\s_2\end{smallmatrix}\bigr]\zeta=k_{[\tau\,h]}$ holds. We obtain the following factorization of \eqref{EQ-Q}
\begin{equation}\label{EQ-FAC-Q}
\begin{tikzcd}[column sep = 1.5em, row sep = 1.2em, ampersand replacement=\&]
 \ker g \ar[rr,>->,"{\bigl[\begin{smallmatrix}-k_g\\\phantom{-}\tau''\end{smallmatrix}\bigr]}"] 
\&\& Y'\oplus\ker h     \ar[rr,"{[-u_2\;u_3]}"] \ar[dr,->>,"r"']
\&\& S \ar[rr,->>,"{\phantom{!!}[\hspace{1pt}1\;\,0\hspace{1pt}]s_1\phantom{!!!!}}"]
\&\& X, \\
\&\&\&\ker{[\hspace{1pt}\tau\;\,h\hspace{1pt}]}\ar[ur,>->,"{\zeta}"']\&\&\&\&
\end{tikzcd}
\end{equation}
where only the commutativity needs still to be checked. In order to do so we use \eqref{EQ-PB-PRF} and in the last step also the commutativity of \eqref{EQ-CONE-TAU} to obtain
\begin{equation}\label{DOUBLE}
\begin{aligned}
[\hspace{1pt}1\;\,0\hspace{1pt}] s_1 [\hspace{1pt}-u_2\;\,u_3\hspace{1pt}] &= \bigl[ -[\hspace{1pt}1\;\,0\hspace{1pt}]s_1u_2\;\;[\hspace{1pt}1\;\,0\hspace{1pt}]s_1u_3\bigr]=\bigl[ -[\hspace{1pt}1\;\,0\hspace{1pt}]\bigl[\begin{smallmatrix}0\\g\end{smallmatrix}\bigr]\;\;[\hspace{1pt}1\;\,0\hspace{1pt}]\bigl[\begin{smallmatrix}0\\0\end{smallmatrix}\bigr]\bigr] = [\hspace{1pt}0\;\;0\hspace{1pt}],\\[2pt]
\bigl[\begin{smallmatrix}[\hspace{1pt}0\;\,1\hspace{1pt}]s_1\\s_2\end{smallmatrix}\bigr][\hspace{1pt}-u_2\;\;u_3\hspace{1pt}] & =  \Bigl[\begin{smallmatrix} -[\hspace{1pt}0\;\,1\hspace{1pt}]s_1u_2 & [\hspace{1pt}0\;\,1\hspace{1pt}]s_1u_3\\-s_2u_2 & s_2u_3\end{smallmatrix}\Bigr] = \Bigl[\begin{smallmatrix} -g & 0\\\phantom{-}\tau' & k_h\end{smallmatrix}\Bigr] = k_{[\tau\,h]}r.
\end{aligned}
\end{equation}
Using the first equation in \eqref{DOUBLE} and the universal property of $\zeta=\ker([\hspace{1pt}1\;\,0\hspace{1pt}]s_1)$ we get a unique map $\xi\colon Y'\oplus\ker h\rightarrow \ker[\hspace{1pt}\tau\;\,h\hspace{1pt}]$ such that $\zeta\xi=[\hspace{1pt}-u_2\;\,u_3\hspace{1pt}]$. Employing first what we noted after \eqref{EQ-ZETA}, then the former equality and finally the second equation in \eqref{DOUBLE}, we obtain
\[
k_{[\tau\,h]}\xi = \bigl[\begin{smallmatrix}[\hspace{1pt}0\;\,1\hspace{1pt}]s_1\\s_2\end{smallmatrix}\bigr]\zeta\xi = \bigl[\begin{smallmatrix}[\hspace{1pt}0\;\,1\hspace{1pt}]s_1\\s_2\end{smallmatrix}\bigr][\hspace{1pt}-u_2\;\,u_3\hspace{1pt}]=k_{[\tau\,h]}r
\]
which, as $k_{[\tau\,h]}$ is monic, implies $\xi=r$. Thus the commutativity of \eqref{EQ-CONE-TAU} yields the commutativity of \eqref{EQ-FAC-Q} and finishes the proof of the first part of (MS2).

\smallskip

\textcircled{2} Now we proceed with the second part of (MS2). Let for this  $\tau\colon Z_h\rightarrow X_f$ and $\alpha\colon Z_h\rightarrow Y_g$ be morphisms with $\tau\in\mathcal{S}_{\CC}$, i.e., we have the commutative diagram
\[
\begin{tikzcd}
\ker f\ar[r,hook,"k_f"]& X'\arrow{r}{f} & X \\[4pt]
\ker h\ar[r,hook,"k_h"]\ar[d,swap,"{\alpha''}"]\ar[u,"\tau''"]& Z' \arrow{r}{h}\arrow{u}{\tau'}\ar[d,swap,"{\alpha'}"] & Z\ar[u,swap,"\tau"]\ar[d,"{\alpha}"]\\[4pt]
\ker g\ar[r,hook,swap,"k_g"]& Y' \arrow{r}[swap]{g} & Y.
\end{tikzcd}
\]
We define $S:=X\oplus Y$, $S':=Z\oplus X'\oplus Y'$ and $u:=\bigl[\begin{smallmatrix}\phantom{-}\tau&f&0\\-\alpha&0&g\end{smallmatrix}\bigr]\colon S'\rightarrow S$, denote by $k_u=[\hspace{1pt}k_1\;k_2\;k_3\hspace{1pt}]^{\T}\colon\ker u\rightarrow S'$ the kernel of $u$ and get an object $S_u\in\hMork(\mathcal{A})$. Next we define maps $\sigma\colon Y_g\rightarrow S_u$ and $\beta\colon X_f\rightarrow S_u$ 
\[
\begin{tikzcd}
\ker f\ar[r,hook,"k_f"]\ar[d,swap,"\beta''"]& X'\arrow{r}{f}\arrow{d}[swap]{\beta'} & X\ar[d,"\beta"]\\[4pt]
\ker u\ar[r,hook,"k_u"]& S' \arrow{r}{u} & S \\[4pt]
\ker g\ar[r,hook,swap,"k_g"]\ar[u,"{\sigma''}"]& Y'\ar[u,"{\sigma'}"] \arrow{r}[swap]{g} & Y\ar[u,swap,"{\sigma}"]
\end{tikzcd}
\]
by letting $\sigma,\sigma',\beta,\beta'$ be the natural inclusions and $\sigma''$, $\beta''$ be the maps induced by the universal property of $\ker u$. We see that the following square
\begin{equation*}
\begin{tikzcd}
Z_h\arrow{r}{\alpha}\arrow{d}[swap]{\tau} & Y_g \arrow{d}{\sigma}\\[4pt]
X_f \arrow{r}[swap]{\beta} & S_u
\end{tikzcd}
\end{equation*}
is commutative in $\hMork(\mathcal{A})$ by using (ii) with $\rho:=[\hspace{1pt}1\;0\;0\hspace{1pt}]^{\T}\colon Z\rightarrow S'$ to conclude that $\beta\tau-\sigma\alpha$ is null-homotopic. As in the previous part we calculate $\cone(\sigma)$ and subtract two null-homotopic complexes leading to
\begin{equation}\label{EQ-Q2}
\ker g\xrightarrow{\sigma''}\ker u\xrightarrow{\bigl[\begin{smallmatrix}k_1\\k_2\end{smallmatrix}\bigr]}Z\oplus X'\xrightarrow{[\hspace{1pt}\tau\;\,f\hspace{1pt}]} X
\end{equation}
which we have to show is acyclic. Since $\cone(\tau)$ is acyclic, we get
\begin{equation}\label{EQ-CONE-TAU-2}
\begin{tikzcd}[column sep = 1.5em, row sep = 1.2em, ampersand replacement=\&]
 \ker h \ar[rr,>->,"{\bigl[\begin{smallmatrix}-k_h\\\phantom{-}\tau''\end{smallmatrix}\bigr]}"] 
\&\& Z'\oplus\ker f     \ar[rr,"{\bigl[\begin{smallmatrix}-h&0\\\phantom{-}\tau'&k_f\end{smallmatrix}\bigr]}"] \ar[dr,->>,"r"']
\&\& Z\oplus X' \ar[rr,->>,"{[\hspace{1pt}\tau\;\,f\hspace{1pt}]}"]
\&\& X. \\
\&\&\&\ker{[\hspace{1pt}\tau\;\,f\hspace{1pt}]}\ar[ur,>->,"{k_{[\tau\,f]}}"']\&\&\&\&
\end{tikzcd}
\end{equation}
Next, we compute
\[
0=[\hspace{1pt}1\;\,0\hspace{1pt}]\bigl[\begin{smallmatrix}0\\0\end{smallmatrix}\bigr]=[\hspace{1pt}1\;\,0\hspace{1pt}]uk_u=[\hspace{1pt}1\;\,0\hspace{1pt}]\bigl[\begin{smallmatrix}\phantom{-}\tau&f&0\\-\alpha&0&g\end{smallmatrix}\bigr]\Bigl[\begin{smallmatrix}k_1\\k_2\\k_3\end{smallmatrix}\Bigr]=[\hspace{1pt}\tau\;\,f\hspace{1pt}]\bigl[\begin{smallmatrix}k_1\\k_2\end{smallmatrix}\bigr]\;\text{ and }\;u\Bigl[\begin{smallmatrix}-h&0\\\phantom{-}\tau'&k_f\\[-1pt]-\alpha'&0\end{smallmatrix}\Bigr]=0.
\]
From the first equation and the universal property of the kernel of $[\hspace{1pt}\tau\;\,f\hspace{1pt}]$ we get a map $\zeta\colon\ker u\rightarrow \ker[\hspace{1pt}\tau\;\,f\hspace{1pt}]$ with $k_{[\tau\,f]}\zeta=\bigl[\begin{smallmatrix}k_1\\k_2\end{smallmatrix}\bigr]$. From the second equation and the universal property of the kernel of $u$ we get a map $v=[\hspace{1pt}v_1\;\,v_2\hspace{1pt}]\colon Z'\oplus\ker f\rightarrow\ker u$ with
\[
\Bigl[\begin{smallmatrix}-h&0\\\phantom{-}\tau'&k_f\\[-1pt]-\alpha'&0\end{smallmatrix}\Bigr]=k_uv=\Bigl[\begin{smallmatrix}k_1v_1 & k_1v_2\\k_2v_1&k_2v_2\end{smallmatrix}\Bigr]=\Bigl[\begin{smallmatrix}k_1\\k_2\\k_3\end{smallmatrix}\Bigr][\hspace{1pt}v_1\;\,v_2\hspace{1pt}]=\Bigl[\begin{smallmatrix}k_1v_1 & k_1v_2\\k_2v_1&k_2v_2\\k_3v_1&k_3v_2\end{smallmatrix}\Bigr].
\]
Using the commutativity of \eqref{EQ-CONE-TAU-2} and the above it follows
\[
k_{[\tau\,f]}r=\bigl[\begin{smallmatrix}-h&0\\\phantom{-}\tau'&k_f\end{smallmatrix}\bigr]=\bigl[\begin{smallmatrix}k_1\\k_2\end{smallmatrix}\bigr][\hspace{1pt}v_1\;\,v_2\hspace{1pt}]=k_{[\tau\,f]}\zeta v
\]
and since $k_{[\tau\,f]}$ is monic we get $r=\zeta v$ and thus $\zeta v$ is a deflation by \eqref{EQ-CONE-TAU-2}. Applying \hypref{R3} we obtain that $\zeta$ is a deflation, too. We claim $\sigma''=\ker\zeta$. Firstly, $k_{[\tau\,f]}\zeta\sigma''=\bigl[\begin{smallmatrix}k_1\\k_2\end{smallmatrix}\bigr]\sigma''=0$ in view of \eqref{EQ-Q2}. Let next $t\colon T\rightarrow\ker u$ be such that $\zeta t=0$. Then
\begin{equation}\label{EQ-3}
 0=k_{[\tau\,f]}\zeta t = \bigl[\begin{smallmatrix}k_1t\\k_2t\end{smallmatrix}\bigr]\;\;\text{ and (thus)}\;\;0=uk_ut=\bigl[\begin{smallmatrix}\phantom{-}\tau{}k_1t+fk_2t\\-\alpha{}k_1t+gk_3t\end{smallmatrix}\bigr]=\bigl[\begin{smallmatrix}0\\gk_3t\end{smallmatrix}\bigr]
\end{equation}
which yields $gk_3t=0$. By the universal property of the kernel of $g$ we get a map $\mu\colon T\rightarrow\ker g$ with $k_3t=k_g\mu$ and employing the first equality in \eqref{EQ-3} we derive
\[
k_u\sigma''\mu=\sigma'k_g\mu=\sigma'k_3t=\Bigl[\begin{smallmatrix}0\\0\\1\end{smallmatrix}\Bigr]k_3t=\Bigl[\begin{smallmatrix}k_1t\\k_2t\\k_3t\end{smallmatrix}\Bigr]=k_ut,
\]
which means $\sigma''\mu=t$ as $k_u$ is monic. As $\sigma'$ is an inclusion we get that $\sigma'k_g=k_u\circ\sigma''$ is monic from whence it follows that $\sigma''$ is monic. This implies that the map $\mu$ is unique and we have established that $\sigma''=\ker\zeta$ holds. Coming back to \eqref{EQ-Q2} we obtain
\begin{equation}\label{EQ-Q2a}
\begin{tikzcd}[column sep = 1.5em, row sep = 1.2em]
\ker g\ar[rr,>->,"{\sigma''}"]&&\ker u\ar[rd,->>,"{\zeta}"']\ar[rr,"{\bigl[\begin{smallmatrix}k_1\\k_2\end{smallmatrix}\bigr]}"]&&Z\oplus X'\ar[rr,->>,"{[\hspace{1pt}\tau\;\,f\hspace{1pt}]}"]&&X\\
&&&\ker{[\hspace{1pt}\tau\;\,f\hspace{1pt}]}\ar[ur,>->,"{k_{[\tau\,f]}}"']&&&&
\end{tikzcd}
\end{equation}
of which we this time showed already the commutativity and whence may conclude that $\cone(\sigma)$ is acyclic.

\smallskip

\textcircled{3} We continue with the first implication of (MS3). Let $\alpha,\beta\colon X_f\rightarrow Y_g$ and $\sigma\colon Z_h\rightarrow X_f$ be morphisms such that $\cone(\sigma)$ is acyclic and $\alpha\sigma=\beta\sigma$ holds in $\hMork(\mathcal{A})$. Consider $\delta:=\alpha-\beta\colon X_f\rightarrow Y_g$, define $U:=Y$, $U':=X\oplus Y'$ and $s:=[\hspace{1pt}\delta\;\,g\hspace{1pt}]\colon U'\rightarrow U$. Then we get a morphism $\tau\colon Y_g\rightarrow U_s$
\[
\begin{tikzcd}
\ker f\ar[r,hook,"k_f"]\ar[d,swap,"\delta''"]& X'\arrow{r}{f}\arrow{d}[swap]{\delta'} & X\ar[d,"\delta"] \\[4pt]
\ker g\ar[r,hook,"k_g"]\ar[d,swap,"{\tau''}"]& Y' \arrow{r}{g}\ar[d,swap,"{\tau'}"]  & Y\ar[d,"{\tau}"]\\[4pt]
\ker s\ar[r,hook,swap,"k_s"]& U'\arrow{r}[swap]{s} & U
\end{tikzcd}
\]
where $\tau$ is the identity, $\tau'$ is the natural inclusion and $\tau''$ is the induced map. Using (ii) with $\rho=\bigl[\begin{smallmatrix}1\\0\end{smallmatrix}\bigr]\colon X\rightarrow X\oplus Y'$ we see that $\tau\delta$ is null-homotopic and thus $\tau\alpha=\tau\beta$ in $\hMork(\mathcal{A})$. To see that $\cone(\tau)$ is acyclic we calculate $\cone(\delta)$
\begin{equation}\label{EQ-CONE-DELTA}
\ker f\xrightarrow{\bigl[\begin{smallmatrix}-k_f\\\phantom{-}\delta''\end{smallmatrix}\bigr]}X'\oplus\ker g\xrightarrow{\bigl[\begin{smallmatrix}-f&0\\\phantom{-}\delta'&k_g\end{smallmatrix}\bigr]}X\oplus{}Y'\xrightarrow{[\hspace{1pt}\delta\;\,g\hspace{1pt}]} Y.
\end{equation}
Applying \cite[Thm IV.4.11(a)]{GM} to the distinguished triangle $X_f\stackrel{\scriptscriptstyle\delta}{\rightarrow} Y_g\stackrel{\scriptscriptstyle\gamma}{\rightarrow}\cone(\delta)\rightarrow\Sigma X_f$, where $\gamma$ is the natural map, we get the exact sequence
\[
\operatorname{LH}^0(X_f)\xrightarrow{\operatorname{LH}^0(\delta)}\operatorname{LH}^0(Y_g)\xrightarrow{\operatorname{LH}^0(\gamma)}\operatorname{LH}^0(\cone(\delta))\longrightarrow\operatorname{LH}^1(X_f).
\]
Employing \eqref{EQ-LH} we see that $\operatorname{LH}^1(X_f)=\cdots\rightarrow0\rightarrow X\rightarrow X\rightarrow 0\rightarrow\cdots$, which is zero in $\mathcal{LH}(\mathcal{A},\CC)$. Moreover we see $\operatorname{LH}^0(X_f)=X_f$, $\operatorname{LH}^0(Y_g)=Y_g$ and $\operatorname{LH}^0(\delta)=\delta$. Further we compute $\operatorname{LH}^0(\cone(\delta))=U_s$ and $\operatorname{LH}^0(\gamma)=\tau$. Therefore, the above is in fact
\[
X_f\xrightarrow{\delta}Y_g\xrightarrow{\tau}U_s\rightarrow0,
\]
whose exactness yields $\tau=\cok\delta$ in the abelian category $\mathcal{LH}(\mathcal{A},\CC)$. By our assumptions at the beginning we have $\delta\sigma=(\alpha-\beta)\sigma=0$ in $\hMork(\mathcal{A})$ and thus also in $\K(\mathcal{A})$. It follows $0=\delta\sigma$ in $\Der(\mathcal{A},\CC)$. Since $\cone(\sigma)$ is acyclic, $\sigma$ is invertible in $\Der(\mathcal{A},\CC)$ and the previous equation gives us $\delta=0$ in $\Der(\mathcal{A},\CC)$ and thus also in $\mathcal{LH}(\mathcal{A},\CC)$. Consequently, $\tau=\cok\delta=\cok0$ is an isomorphism in $\mathcal{LH}(\mathcal{A},\CC)$ and, due to the thickness of $\Ac(\mathcal{A},\CC)$, $\cone(\tau)$ is acyclic. 

\smallskip

\textcircled{4} It remains to check the backwards direction of (MS3). For this, let $\alpha,\beta\colon X_f\rightarrow Y_g$ and $\tau\colon Y_g\rightarrow U_s$ with $\cone(\tau)$ acyclic and $\tau\alpha=\tau\beta$ in $\hMork(\mathcal{A})$. Put $\delta:=\alpha-\beta$. Then $\tau\delta$ is null-homotopic, i.e., there exists $\rho\colon X\rightarrow U'$ with $\tau\delta=s\rho$. The complex $\cone(\tau)$ being acyclic means that we have the factorization
\begin{equation}\label{EQ-CONE-TAU-3}
\begin{tikzcd}[column sep = 1.5em, row sep = 1.2em, ampersand replacement=\&]
 \ker g \ar[rr,>->,"{\bigl[\begin{smallmatrix}-k_g\\\phantom{-}\tau''\end{smallmatrix}\bigr]}"] 
\&\& Y'\oplus\ker s     \ar[rr,"{\bigl[\begin{smallmatrix}-g&0\\\phantom{-}\tau'&k_s\end{smallmatrix}\bigr]}"] \ar[dr,->>,"r"']
\&\& Y\oplus U' \ar[rr,->>,"{[\hspace{1pt}\tau\;\,s\hspace{1pt}]}"]
\&\& U. \\
\&\&\&\ker{[\hspace{1pt}\tau\;\,s\hspace{1pt}]}\ar[ur,>->,"{k_{[\tau\,s]}}"']\&\&\&\&
\end{tikzcd}
\end{equation}
By the above we have $[\hspace{1pt}\tau\;\,s\hspace{1pt}]\bigl[\begin{smallmatrix}\phantom{-}\delta\\-\rho\end{smallmatrix}\bigr]=\tau\delta-s\rho=0$ and may use the universal property of the kernel of $[\hspace{1pt}\tau\;\,s\hspace{1pt}]$ to obtain a unique map $\zeta\colon X\rightarrow\ker[\hspace{1pt}\tau\;\,s\hspace{1pt}]$ with $k_{[\tau\,s]}\zeta=\bigl[\begin{smallmatrix}\phantom{-}\delta\\-\rho\end{smallmatrix}\bigr]$. Next we form the pullback of $r$ along $\zeta$ and obtain by \hypref{R2} a deflation $\sigma$ and a map $t=\bigl[\begin{smallmatrix}t_1\\t_2\end{smallmatrix}\bigr]$:
\begin{equation}\label{PB-XI}
\begin{tikzcd}[column sep=2em]
Z\ar[d,swap,"t"]\ar[r,->>,"\sigma"]\commutes[\mathrm{PB}]{dr} & X \ar[d,"\zeta"]\\[4pt]
Y'\oplus\ker s \ar[r,->>,swap,"r"] & \ker[\hspace{1pt}\tau\;\,s\hspace{1pt}].
\end{tikzcd}
\end{equation}
To complete the definition of $\sigma\colon Z_h\rightarrow X_f$ we form the pullback of $f$ along $\sigma$ and let $\sigma''$ be the natural map\,---\,which is an isomorphism due to the pullback property
\begin{equation}\label{PB-XI-2}
\begin{tikzcd}
\ker h\ar[r,hook,"k_h"]\ar[d,swap,"\sigma''"]\ar[d,,"\sim"]&Z'\ar[d,->>,swap,"\sigma'"]\ar[r,"h"]\commutes[\mathrm{PB}]{dr} & Z \ar[d,->>,"\sigma"]\\[4pt]
\ker f\ar[r,hook,swap,"k_f"]&X' \ar[r,->,swap,"f"] & X.
\end{tikzcd}
\end{equation}
To show that $\delta\sigma$ is null-homotopic we use (ii) with $\rho:=-t_1$. Then \eqref{EQ-CONE-TAU-3}, \eqref{PB-XI} and the defining property of $\zeta$ yield
\[
\bigl[\begin{smallmatrix}-gt_1\\\tau't_1+k_st_2\end{smallmatrix}\bigr]=\bigl[\begin{smallmatrix}-g&0\\\phantom{-}\tau'&k_s\end{smallmatrix}\bigr]\bigl[\begin{smallmatrix}t_1\\t_2\end{smallmatrix}\bigr]=k_{[\tau\,s]}rt=k_{[\tau\,s]}\zeta\sigma=\bigl[\begin{smallmatrix}\delta\\-\rho\end{smallmatrix}\bigr]\sigma=\bigl[\begin{smallmatrix}\phantom{-}\delta\sigma\\-\rho\sigma\end{smallmatrix}\bigr]
\]
and thus in particular $g(-t_1)=\delta\sigma$. It remains to show that $\cone(\sigma)$ is acyclic. Calculating the cone and subtracting the null-homotopic complex $\cdots\rightarrow0\rightarrow\ker h\stackrel{\hspace{-3pt}\scriptscriptstyle\sim}{\rightarrow}\ker f\rightarrow 0\rightarrow\cdots$ we get
\[
Z'\xrightarrow{\bigl[\begin{smallmatrix}-h\\\phantom{-}\sigma'\end{smallmatrix}\bigr]}Z\oplus X'\xrightarrow{[\hspace{1pt}\sigma\;\,f\hspace{1pt}]}X
\]
which is the squarefold of the pullback in \eqref{PB-XI-2}. Using \cite[dual of Prop 5.7]{BC13} or \cite[Prop 3.7(2)]{HR19a} yields that the sequence is a conflation which finishes the proof of (MS3).

\medskip

Finally we consider the functor $\psi\colon\hMork(\mathcal{A})[\mathcal{S}_{\CC}^{-1}]\rightarrow\mathcal{LH}(\mathcal{A},\CC)$ that extends objects of the category $\hMork(\mathcal{A})$ by zero to complexes, and fractions of maps  by zero to fractions of chain maps. By our previous steps, $\psi$ is well-defined and essentially surjective.

\smallskip

We prove fullness using left fractions. Let thus $X_f\stackrel{\hspace{-3.5pt}\alpha}{\longrightarrow}Z^{\bullet}\stackrel{\sigma\hspace{-2pt}}{\longleftarrow}Y_g$ with $X_f,Y_g\in\hMork(\mathcal{A})$, $Z^{\bullet}\in\C(\mathcal{A})$ and chain maps $\alpha,\sigma$ with $\sigma\in\mathcal{N}(\mathcal{A},\CC)$ be a morphism in the left heart. Denoting by $p$ and $i$ the natural maps we obtain the commutative diagram
\begin{equation*}
\begin{tikzcd}[column sep=22pt, row sep=24pt]
& Z^{\bullet}\ar[d,"p"] & \\
X_f\ar[r,swap, "p\alpha"]\ar[ru,"\alpha"]\ar[rd,swap,"\operatorname{LH}^0(\alpha)"]&\tau_{\operatorname{L}}^{\scriptscriptstyle\geqslant0}Z^{\bullet}&\ar[l,"p\sigma"]\ar[ld, "\operatorname{LH}^0(\sigma)"]\ar[lu,swap,"\sigma"]Y_g\\
&\operatorname{LH}^0(Z^{\bullet})\ar[u,swap,"i"]
\end{tikzcd}
\end{equation*}
in which $\operatorname{LH}^0(Z^{\bullet})=(\ker d_Z^{\hspace{1pt}-1}\hookrightarrow Z^{-1}\rightarrow\ker d_Z^{\hspace{1pt}0})\in\hMork(\mathcal{A})$, $\operatorname{LH}^0(\sigma)\in\mathcal{S}_{\CC}$ and thus the bottom represents a roof in $\hMork(\mathcal{A})[\mathcal{S}_{\CC}^{-1}]$. Since $\sigma\in\mathcal{N}(\mathcal{A},\CC)$, we have $Z^{\bullet}\in\mathcal{LH}(\mathcal{A},\CC)\subseteq\Der^{\geqslant0}_{\operatorname{L}}(\mathcal{A},\CC)$ and thus the map $p$ is a quasiisomorphism.  It follows $p\sigma\in\mathcal{N}(\mathcal{A},\CC)$ which establishes that $\sigma^{-1}\alpha$ has a preimage under $\psi$.

\smallskip

For faithfulness assume that the roof $X_f\stackrel{\hspace{-3.5pt}\alpha}{\longrightarrow}Z_h\stackrel{\sigma\hspace{-2pt}}{\longleftarrow}Y_g$ is zero in $\mathcal{LH}(\mathcal{A},\CC)$. Then there exists a quasi-isomorphisms $\beta\colon A^{\bullet}\rightarrow X_f$ with $\alpha\beta=0$ in $\K(\mathcal{A})$. Let $(\rho^{n})_{n\in\ZZ}$ be the corresponding homotopy maps, in particular it holds $\alpha\beta^0=\rho^{-1}d_A^{\hspace{1pt}0}+h\rho^0$. Applying $\operatorname{LH}^0$ we get $\operatorname{LH}^0(\beta)\in\mathcal{S}_{\CC}$ and
\[
\begin{tikzcd}
\operatorname{LH}^0(A^{\bullet})\ar[d, "{\operatorname{LH}^0(\beta)}"]&\ker d_A^{\hspace{1pt}-1}\ar[r,hook,"i^{-2}"]\ar[d,dashed,swap]& A^{-1}\arrow{r}{p^{-1}}\arrow{d}[swap]{\beta^{-1}} & \ker d_A^{\hspace{1pt}0}\ar[d,"\beta^{0}i^{-1}"] \\[4pt]
X_f\ar[d,"\alpha"]&\ker f\ar[r,hook,"k_f"]\ar[d,swap,"{\alpha''}"]& X' \arrow{r}{f}\ar[d,swap,"{\alpha'}"]  & X\ar[d,"{\alpha}"]\\[4pt]
Z_h&\ker h\ar[r,hook,swap,"k_h"]& Z'\arrow{r}[swap]{h} & Z
\end{tikzcd}
\]
where the dashed morphism is induced by the universal property of the kernel of $f$. We then define $\rho\colon\ker d_A^{\hspace{1pt}0}\rightarrow Z'$, $\rho:=\rho^0i^{-1}$, and obtain $h\rho=\alpha\beta^0i^{-1}$ showing that $\alpha\operatorname{LH}^0(\beta)=0$ holds in $\hMork(\mathcal{A})$, which establishes that the roof we started with is zero on $\hMork(\mathcal{A})[\mathcal{S}_{\CC}^{-1}]$.

\smallskip

That $\mathcal{S}_{\CC}\subseteq\hMork(\mathcal{A})$ is saturated follows from the equivalence and the fact that $\mathcal{N}(\mathcal{A},\CC)\subseteq\K(\mathcal{A})$ is saturated. 
\end{proof}

\begin{rmk}\label{MS-RMK} We add the following comments on the above proof.\vspace{3pt}
\begin{myitemize}

\item[(i)] Assume that $\mathcal{A}$ is essentially small; compare Remark \ref{SIZE}. The equivalence in Theorem \ref{NEW-LH} and the properties of $\mathcal{S}_{\CC}$ can then also be derived by functor category methods as used in \cite[Section 3.5]{HKRW}: Let $\operatorname{mod}(\mathcal{A})$ be the category of finitely presented contravariant $\operatorname{Ab}$-valued functors and let $\mathbb{Y}\colon\mathcal{A}\rightarrow\operatorname{mod}(\mathcal{A})$, $\mathbb{Y}(X)=\Hom_{\mathcal{A}}(\,\cdot\,,X)$. Observe firstly that
\[
\Gamma\colon\hMork(\mathcal{A})\rightarrow\operatorname{mod}(\mathcal{A}),\;\;X_f\mapsto\cok\mathbb{Y}(f)
\]
is an equivalence. Next, denote by $\bar{\phi}\colon\operatorname{mod}(\mathcal{A})\rightarrow\mathcal{LH}(\mathcal{A},\CC)$ the extension of $\phi$ which happens to be exact, see \cite[Prop 3.9 together with Prop 2.2]{HKRW}. Using additionally \cite[Cor 3.10]{HKRW} one gets the natural isomorphisms
\[
\bar{\phi}\hspace{1pt}\Gamma(X_f)\cong \cok_{\mathcal{LH}(\mathcal{A},\CC)}(f)\cong X_f^{\bullet},
\]
implying that $\bar{\phi}\hspace{1pt}\Gamma$ is naturally isomorphic to $\hMork(\mathcal{A})\rightarrow\mathcal{LH}(\mathcal{A},\CC)$, $X_f\mapsto X_f^{\bullet}$. Let finally $\operatorname{eff}(\mathcal{A},\CC)$ be the subcategory of effaceable functors which is Serre, see \cite[Dfn 3.18 and Prop 3.19]{HKRW}. Observe that the localization functor $\operatorname{mod}\mathcal{A}\rightarrow\operatorname{mod}\mathcal{A}/\operatorname{eff}(\mathcal{A},\CC)$ coincides with $\bar{\phi}$ if we identify $\operatorname{mod}\mathcal{A}/\operatorname{eff}(\mathcal{A},\CC)\cong\mathcal{LH}(\mathcal{A},\CC)$ via \cite[Thm 3.20]{HKRW}. Standard facts about the localization with respect to Serre subcategories yield
\[
\alpha\in\mathcal{S}_{\CC}\,\Longleftrightarrow\, \bar{\phi}\hspace{1pt}\Gamma(\alpha) \text{ is invertible}\;\Longleftrightarrow\;\ker\Gamma(\alpha),\,\cok\Gamma(\alpha)\in\operatorname{eff}(\mathcal{A},\CC)
\]
where $\Sigma_{\operatorname{eff}(\mathcal{A},\CC)}=\{\beta\in\mod(\mathcal{A})\:|\:\ker\beta,\,\cok\beta\in\operatorname{eff}(\mathcal{A},\CC)\}$ is a saturated multiplicative system \cite[Ex 8.12]{KS}. Under the equivalence $\Gamma$ the two classes $\mathcal{S}_{\CC}$ and $\Sigma_{\operatorname{eff}(\mathcal{A},\CC)}$ correspond, from which the claimed properties of $\mathcal{S}_{\CC}$ as well as the equivalence follow.

\vspace{3pt}

\item[(ii)]  Although we used (following \cite{HKRW} and \cite{S}) the unbounded derived category to define the t-structure and the left heart, we could have also done everything inside $\Der^{\bb}(\mathcal{A},\CC)$ from the beginning. Thus, a `minimal' set-theoretic assumption that guarantees that everything so far works out is that $\Der^{\bb}(\mathcal{A},\CC)$ is well-defined or `locally small'.

\vspace{3pt}

\item[(iii)] The direct proof of the axioms (MS2) and (MS3) that we gave above provides explicit formulas necessary to actually calculate with the morphisms of $\mathcal{LH}(\mathcal{A},\CC)$, i.e., for composition and cancellation of fractions.



\vspace{3pt}

\item[(iv)] At the very end of the proof above we used that a morphism $\alpha\colon X_f\rightarrow Y_g$ such that $\alpha''\colon\ker f\rightarrow\ker g$ is an isomorphism belongs to $\mathcal{S}_{\CC}$ iff 
\begin{equation}\label{SQRFLD}
X'\xrightarrow{\bigl[\begin{smallmatrix}-f\\\phantom{-}\alpha'\end{smallmatrix}\bigr]}X\oplus Y'\xrightarrow{[\hspace{1pt}\alpha\;\;g\hspace{1pt}]}Y\in\CC.
\end{equation}
This may be rephrased as the \eqref{C-PUL} being a `$\CC$-pulation', see also Remark \ref{REM}:
\begin{equation}\label{C-PUL}
\begin{tikzcd}[column sep =21pt, row sep =17pt]
X'\arrow{r}{f}\arrow{d}[swap]{\alpha'} & X \arrow{d}{\alpha}\\[4pt]
Y' \arrow{r}[swap]{g} & Y.
\end{tikzcd}
\end{equation}

\item[(v)] If the diagram \eqref{C-PUL} is merely a pullback, the assumption of (iv) is satisfied: $\alpha''$ is an isomorphism. On the other hand, for $\mathcal{A}=\mathsf{Vect}(\KK)$ the morphism
\begin{equation*}
\begin{tikzcd}[column sep =35pt, row sep =20pt, ampersand replacement=\&]
\KK\arrow[hook]{r}{\bigl[\begin{smallmatrix}0\\1\end{smallmatrix}\bigr]}\ar{d}[swap]{0}\& \KK^2\arrow{r}{\bigl[\begin{smallmatrix}1&0\\0&0\end{smallmatrix}\bigr]}\arrow{d}[swap]{0} \& \KK^2\arrow{d}{[\hspace{1pt}0\;\;1\hspace{1pt}]} \\[4pt]
0\ar[r,hook,swap,"0"]\& 0 \arrow{r}[swap]{0} \& \KK
\end{tikzcd}
\end{equation*}
has acyclic cone without that the squarefold of the right hand side, namely the sequence $\KK^2\rightarrow \KK^2\rightarrow\KK$ as in \eqref{SQRFLD}, is a kernel-cokernel pair.
\end{myitemize}
\end{rmk}

Next we will give explicit formulae for the computation of kernels and cokernels in the left heart. Notice that below it is sufficient to calculate kernels and cokernels in $\hMork(\mathcal{A})$ since $\ker(\sigma^{-1}\alpha)\cong\ker\alpha$ and $\cok(\sigma^{-1}\alpha)\cong\cok\alpha$ hold for fractions $\sigma^{-1}\alpha\in\hMork(\mathcal{A})[\mathcal{S}_{\CC}^{-1}]$.

\smallskip

For part (iii) let us note that, using Theorem \ref{NEW-LH}, the embedding $\phi\colon\mathcal{A}\rightarrow\hMork(\mathcal{A})$  sends an object $X$ to $X_0=\bigl(\hspace{1pt}0\rightarrow0\rightarrow X\hspace{1pt}\bigr)$ and a morphism $f\colon X\rightarrow Y$ to $(0,0,f)\colon X_0\rightarrow Y_0$.

\begin{thm}\label{KER-COK-PROP} Let $(\mathcal{A},\CC)$ be a strongly deflation-exact category with kernels and let $\alpha\colon X_f\rightarrow Y_g$ be a morphism that we regard in $\mathcal{LH}(\mathcal{A},\CC)\cong\hMork(\mathcal{A})[\mathcal{S}^{-1}]$. Then\vspace{2pt}

\begin{myitemize}
\item[(i)] The kernel of $\alpha$ is given by
\begin{equation*}
\begin{tikzcd}[column sep =25pt, row sep =16pt, ampersand replacement=\&]
\ker s\arrow[hook]{r}{k_s}\ar{d}[swap]{\gamma''}\& X'\oplus\ker g\arrow{r}{s}\arrow{d}[swap]{\gamma'} \& \ker[\hspace{1pt}\alpha\,-\!g\hspace{1pt}]\arrow{d}{\gamma} \\[4pt]
\ker f\ar[r,hook,swap,"k_f"]\& X' \arrow{r}[swap]{f} \& X
\end{tikzcd}
\end{equation*}
where $U=\ker[\hspace{1pt}\alpha\;-\!g\hspace{1pt}]$ is the pullback of $g$ along $\alpha$, $\gamma=[\hspace{1pt}1\;\,0\hspace{1pt}]k_{[\alpha\,-g]}$ is the corresponding map, $U'=X'\oplus\ker g$, $\gamma'=[\hspace{1pt}1\;\,0\hspace{1pt}]$ and $s$ is the unique map with $\gamma{}s=[\hspace{1pt}f\;\,0\hspace{1pt}]$ and $[\hspace{1pt}0\;1\hspace{1pt}]k_{[\alpha\,-g]}s=[\hspace{1pt}\alpha'\;k_g\hspace{1pt}]$, which exists by the pullback property.
\vspace{2pt}

\item[(ii)] The cokernel of $\alpha$ is given by
\begin{equation*}
\begin{tikzcd}[column sep =30pt, row sep =16pt, ampersand replacement=\&]
\ker g\arrow[hook]{r}{k_g}\ar{d}[swap]{\delta''}\& Y'\arrow{r}{g}\arrow{d}[swap]{\delta'} \& Y\arrow{d}{\delta} \\[4pt]
\ker[\hspace{1pt}\alpha\;\,g\hspace{1pt}]\ar[r,hook,swap,"k_{[\alpha\,g]}"]\& X\oplus Y' \arrow{r}[swap]{[\hspace{1pt}\alpha\;\,g\hspace{1pt}]} \& Y
\end{tikzcd}
\end{equation*}
where $\delta=1$ and $\delta'=\bigl[\begin{smallmatrix}0\\1\end{smallmatrix}\bigr]$.
\vspace{2pt}

\vspace{2pt}

\item[(iii)] For $f\colon X\rightarrow Y$ in $\mathcal{A}$ we obtain as special cases of (i) and (ii):
\begin{equation*}
\begin{aligned}
\ker(X_0\xrightarrow{\scriptscriptstyle(0,0,f)}Y_0)&=\bigl((\ker f)_0\xrightarrow{\scriptscriptstyle(0,0,k_{\scalebox{0.4}{$f$}})}X_0\bigr),\\
 \cok(X_0\xrightarrow{\scriptscriptstyle(0,0,f)}Y_0)&=\bigl(Y_0\xrightarrow{\scriptscriptstyle(0,0,1)}(\ker f\stackrel{\scriptscriptstyle k_{\scalebox{0.4}{$f$}}}{\hookrightarrow} X\stackrel{\scriptscriptstyle f}{\rightarrow}Y)\bigr).
\end{aligned}
\end{equation*}
\end{myitemize} 
\end{thm}
\begin{proof}(i) We take cohomology of the distinguished triangle  $X_f\xrightarrow{\hspace{-2pt}\scriptscriptstyle\alpha}Y_g\xrightarrow{\hspace{-2pt}\scriptscriptstyle\nu}\cone(\alpha)\xrightarrow{\hspace{-2pt}\scriptscriptstyle\mu}\Sigma X_f$ and obtain the following exact sequence which identifies the kernel of $\alpha$ in $\mathcal{LH}(\mathcal{A},\CC)$:
\[
0\longrightarrow\operatorname{LH}^{-1}(\cone(\alpha))\xrightarrow{\operatorname{LH}^{0}(\Sigma^{-1}\mu)}X_f\xrightarrow{\;\;\alpha\;\;}Y_g.
\]
In order to compute $\operatorname{LH}^{-1}(\cone(\alpha))$ and $\operatorname{LH}^{0}(\Sigma^{-1}\mu)$ explicitly, we first perform a shift by one to the right on $\mu=(1,[1\;0],[1\;0],0)\colon\cone(\alpha)\rightarrow\Sigma X_f$. With suitable sign changes and setting $D:=\bigl[\begin{smallmatrix}1&\phantom{-}0\\0&-1\end{smallmatrix}\bigr]$ we get the solid part of the following diagram with $X\oplus Y'$ in degree zero and $Dk_{[\alpha\,-g]}$ being the kernel of $-[\hspace{1pt}\alpha\;\,g]\hspace{1pt}$:
\begin{equation*}
\begin{tikzcd}[column sep = 1.5em, row sep = 1.25em, ampersand replacement=\&]
\Sigma^{-1}\cone(\alpha)\ar[dd,swap,"\Sigma^{-1}\mu"]\& \ker f\ar[dd,swap,"1"] \ar[rr,->,"{\bigl[\begin{smallmatrix}\phantom{-}k_f\\-\alpha''\end{smallmatrix}\bigr]}"] 
\&\& X'\oplus\ker g\ar[dd,"{[\hspace{1pt}1\;\,0\hspace{1pt}]}"]\ar[rr,"{F:=\bigl[\begin{smallmatrix}\phantom{-}f&0\\-\alpha'&-k_g\end{smallmatrix}\bigr]}"] \ar[dr,->,"s"']
\&\& X\oplus Y'\ar[dd,"{[\hspace{1pt}1\;\,0\hspace{1pt}]}"] \ar[rr,->,"{-[\hspace{1pt}\alpha\;\,g\hspace{1pt}]}"]
\&\& Y\ar[dd,"0"]\\
\&\&\ker F\ar[ld,dashed, "\gamma''"]\ar[ru,hook, "k_F"'] \&\&\ker{[\hspace{1pt}\alpha\;-\!g\hspace{1pt}]}\ar[rd,dashed,swap,"\gamma"]\ar[ur,hook,"{Dk_{[\alpha\,\sm{}g]}}"']\&\&\&\&\\
X_f\&\ker f\ar[rr,hook,swap,"k_f"] \&\& X'\ar[rr,swap,"f"] \&\& X\ar[rr,swap,"0"] \&\& 0.
\end{tikzcd}
\end{equation*}
We then put $\gamma:=[\hspace{1pt}1\;\,0\hspace{1pt}]k_{[\alpha\,-\!g]}$, $\gamma':=[\hspace{1pt}1\;\,0\hspace{1pt}]$ and obtain $\gamma''$ from the universal property of the kernel of $f$. Observing that $k_F=\ker s$ holds finishes this part of the proof.

\smallskip

(ii) We use the same triangle as in (i) and consider the exact sequence
\[
X_f\xrightarrow{\;\;\alpha\;\;}Y_g\xrightarrow{\operatorname{LH}^{0}(\nu)}\operatorname{LH}^{0}(\cone(\alpha))\longrightarrow0
\]
and observe that $\operatorname{LH}^{0}(\cone(\alpha))$ is precisely the second row in (ii) and that $\operatorname{LH}^{0}(\nu)$ coincides with the map $\delta$.
\end{proof}

We have to come back to the notion of acyclicity of a complex over a non-abelian category \emph{in a single degree}. In view of \eqref{EQ-LH} the following is a very natural definition.

\begin{dfn}\label{DFN-AC-N} Let $(\mathcal{A},\CC)$ be strongly deflation-exact with kernels. We say that a cochain complex $C^{\bullet}$ over $\mathcal{A}$ is \emph{$\CC$-acyclic in degree} $n\in\ZZ$ if $(i^{\hspace{0.5pt}n-2},p^{\hspace{0.5pt}n-1})$ is a conflation. This is equivalent to $\operatorname{LH}^n(C^{\bullet})=0$ in $\mathcal{LH}(\mathcal{A},\CC)$.\hfill\diam
\end{dfn}

The above extends the classical situation of an abelian category with the usual cohomology $\operatorname{H}^n$. We emphasize that the definition of acyclicity in a fixed degree used by \cite[Dfn 2.14]{HKRW}, \cite[p.~691]{Keller96} and \cite[Dfn 7.1]{BC13} is \emph{stronger}, cf.~Remark \ref{FIN-RMK}(iii), and that other definitions, e.g., \cite[Dfn 8.8]{Buehler}, \cite[p.~106]{KrauseBuch} or \cite[Dfn 1.3(1)]{LLS26} can be found in the literature. Nevertheless, a complex being acyclic \emph{in all degrees} is invariant under all the definitions just listed.

\smallskip

Let us now strengthen our assumptions and study strongly deflation-exact categories in which every kernel is an inflation. Then the notion of degree-wise acyclicity can be characterized as follows.

\begin{lem}\label{LEM-AC} Let $(\mathcal{A},\CC)$ be deflation-exact with admissible kernels. Then, for $C^{\bullet}\in\emph{\C}(\mathcal{A})$ and $n\in\ZZ$ the following are equivalent:\vspace{3pt}
\begin{myitemize}
\item[(i)] $\operatorname{LH}^n(C^{\bullet})=0$.

\vspace{0pt}

\item[(ii)] $\ker d^{\hspace{1pt}n-1}\stackrel{i^{n-2}}{\longrightarrow}C^{\hspace{0.5pt}n-1}\stackrel{p^{n-1}}{\longrightarrow}\ker d^{\hspace{1pt}n}\in\CC$.

\vspace{3pt}

\item[(iii)] The differential $d^{\hspace{1pt}n-1}$ has image (and coimage) and $\coim d^{\hspace{1pt}n-1} = \im d^{\hspace{1pt}n-1}=\ker d^{\hspace{1pt}n}$.
\vspace{3pt}

\item[(iv)] The map $p^{n-1}\colon C^{\hspace{0.5pt}n-1}\rightarrow \ker d^{\hspace{1pt}n}$ is a coimage of $d^{\hspace{1pt}n-1}$.

\end{myitemize}
\end{lem}
\begin{proof} (i)\;$\Leftrightarrow$\:(ii) holds by \eqref{EQ-LH}.

\vspace{3pt}

Before we continue let us mention that our general assumptions imply that every morphism $f\colon X\rightarrow Y$ in $\mathcal{A}$ has a coimage and that the coimage map $c\colon X\rightarrow\cok\ker f$ is a deflation.

\vspace{3pt}

(ii)\;$\Rightarrow$\:(iii) We have $p^{n-1}=\cok i^{\hspace{0.5pt}n-2} =\cok \ker d^{\hspace{1pt}n-1}$ which shows that $p^{n-1}\colon C^{\hspace{0.5pt}n-1}\rightarrow\ker d^{\hspace{1pt}n}$ is the coimage of $d^{\hspace{1pt}n-1}$. The kernel $i^{\hspace{0.5pt}n-1}\colon \ker d^{\hspace{1pt}n}\rightarrow C^{\hspace{0.5pt}n}$ is admissible by our general assumptions and thus has a cokernel $q$. Using that $p^{n-1}$ is epic we see that $q=\cok d^{\hspace{1pt}n-1}$, hence $i^{\hspace{0.5pt}n-1}=\ker q = \im d^{\hspace{1pt}n-1}$.

\vspace{3pt}

(iii)\;$\Rightarrow$\:(iv) Specialization.

\vspace{2pt}

(iv)\;$\Rightarrow$\:(ii) By the general assumptions we have $\ker d^{\hspace{1pt}n-1}\stackrel{i^{n-2}}{\longrightarrow}C^{n-1}\stackrel{\!c}{\longrightarrow}\coim d^{\hspace{1pt}n-1}\in\CC$. Due to (iv) this is the same sequence as in (ii).
\end{proof}
 
A morphism for which coimage and image exist and coincide is often called \emph{strict}; condition (iii) above is thus precisely how \cite[Dfns 1.2.10 and 1.1.4]{S} defined the notion of a complex being  \emph{strictly exact in degree $n$} for quasiabelian categories. If, on the other hand, in $(\mathcal{A},\CC)$ there exists a kernel-cokernel pair $X\rightarrow Y\rightarrow Z$  which is not a conflation, then $\cdots\rightarrow 0\rightarrow Y\rightarrow Z\rightarrow0\rightarrow\cdots$ satisfies (iii) in the degree where $Z$ is located but not (ii).

\smallskip

Under the strengthened assumption that $(\mathcal{A},\CC)$ has admissible kernels, the cohomology functor $\operatorname{LH}^n\colon\Der(\mathcal{A},\CC)\rightarrow\mathcal{LH}(\mathcal{A},\CC)$ looks like
\begin{equation}
\operatorname{LH}^n(C^{\bullet})=\cdots\rightarrow0\rightarrow\coim d^{\hspace{1pt}n-1}\hookrightarrow\ker d^{\hspace{1pt}n}\rightarrow0\rightarrow\cdots
\end{equation}
with $\ker d^{\hspace{1pt}n}$ in degree $0$. Every object in $\mathcal{LH}(\mathcal{A},\CC)$ is isomorphic to a complex of the form $\cdots\rightarrow0\rightarrow X'\stackrel{\scriptscriptstyle f}{\hookrightarrow}X\rightarrow 0\rightarrow\cdots$ with $X$ in degree $0$, see \cite[Prop 5.1]{HKRW}, and the properties of the embedding $\phi\colon(\mathcal{A},\CC)\rightarrow\mathcal{LH}(\mathcal{A},\CC)$ improve as follows.

\begin{thm}\label{PHI-IMPROV}\cite[Prop 5.2]{HKRW} Let $(\mathcal{A},\CC)$ be a deflation-exact category with admissible kernels and read $\mathcal{A}\subseteq\mathcal{LH}(\mathcal{A},\CC)$ via the embedding $\phi$.\vspace{2pt}
\begin{myitemize}
\item[(i)] The category $\mathcal{A}\subseteq\mathcal{LH}(\mathcal{A},\CC)$ is closed under subobjects; more precisely if $X\hookrightarrow Z$ is a monomorphism in $\mathcal{LH}(\mathcal{A},\CC)$ with $Z\in\mathcal{A}$, then $X\in\mathcal{A}$, too.

\vspace{2pt} 

\item[(ii)] For every $Z\in\mathcal{LH}(\mathcal{A},\CC)$ there exists a short exact sequence $X\rightarrow Y\rightarrow Z$ in $\mathcal{LH}(\mathcal{A},\CC)$ with $X,Y\in\mathcal{A}$. \hfill\diam{}
\end{myitemize}
\end{thm}

Also the description of the left heart improves.

\begin{thm}\label{THM-HEART-2}\cite[Thm 1.4 together with Prop 4.12]{HKRW} Let $(\mathcal{A},\CC)$ be a deflation-exact category with admissible kernels. Then $\mathcal{LH}(\mathcal{A},\CC)\cong \hMon(\mathcal{A})[\mathcal{S}^{-1}_{\CC}]$. The objects of $\hMon(\mathcal{A})$ are monomorphisms $X_f=(f\colon X'\hookrightarrow X)$ in $\mathcal{A}$, a morphism $\alpha\colon X_f\rightarrow Y_g$ is represented by a commutative square
\begin{equation*}
\begin{tikzcd}
X'\ar[r, hook, "f"]\arrow{d}[swap]{\alpha'} & X \arrow{d}{\alpha}\\[4pt]
Y' \ar[r, hook, swap, "g"] & Y
\end{tikzcd}
\end{equation*}
modulo the equivalence relation $\alpha\sim\beta$ iff $\exists\:\rho\colon X\rightarrow Y'\colon \alpha-\beta=g\rho$. The class $\mathcal{S}_{\CC}$ consists of all morphisms that are pulations and it is a multiplicative system in $\hMon(\mathcal{A})$.
\hfill\diam
\end{thm}

\begin{rmk}\label{REM} We add the following comments.\vspace{3pt}

\begin{myitemize}

\item[(i)] Under the assumptions of Theorem \ref{THM-HEART-2} the category $\hMon(\mathcal{A})$ is already preabelian by \cite[Prop 4, Lem 11]{Wegner17} and \cite[Prop 4.10]{HKRW}.

\vspace{2pt}

\item[(ii)] The proof of \cite[Prop 4]{Wegner17} gives explicit formulae for kernels and cokernels. In the case of kernels the formula of Theorem \ref{KER-COK-PROP}(i) specializes to the one given in \cite[Prop 4]{Wegner17}. In the case of cokernels (and $\ran =\coim$, see \cite[Lem 11]{Wegner17}) this is not true as in general the map $[\hspace{1pt}\alpha\;\,g\hspace{1pt}]$ in Theorem \ref{KER-COK-PROP}(ii) need not be monic merely because $f$ and $g$ are so.

\vspace{2pt}

\item[(iii)] The assumptions on $(\mathcal{A},\CC)$ imply that all pulations are $\CC$-pulations in the sense mentioned in Remark \ref{MS-RMK}(iv).

\vspace{2pt}

\item[(iv)] While \cite{HKRW} use the functor category $\operatorname{mod}(\mathcal{A})$ to prove Theorem \ref{THM-HEART-2} and assume that $\mathcal{A}$ is essentially small in doing that, it is possible to give a proof that only uses the (bounded) derived category by adapting \cite[Cor 1.2.21]{S}. Thus, also Theorem \ref{THM-HEART-2} may be used under our general assumptions on size as outlined in Remark \ref{SIZE}.

\end{myitemize}
\end{rmk}

Dualization of the above gives analogous results for strongly inflation-exact categories with (admissible) cokernels. We will not write out everything explicitly, but mention the following, in particular to fix our notation for the next sections. Notice that, despite dualization, we stick to \emph{co}chain complexes.

\begin{rmk}\label{INF-RMK} Let $(\mathcal{A},\CC)$ be a strongly inflation-exact category with cokernels. Then

\begin{myitemize}

\item[(i)] \begin{tikzcd}[column sep=17pt,
row sep=1pt,
column 1/.style={
    nodes={
      text width=4cm,
      align=right
    }
  }]
\mathmakebox[1.3cm][r]{\tau_R^{\leqslant n}C^{\bullet}}
  &[-20pt]=
  &[-20pt] \cdots\arrow[r]
  & C^{\hspace{0.5pt}n} \arrow[r]
  & C^{\hspace{0.5pt}n+1} \arrow[r, "p^{n+1}"]
  & \cok d^{\hspace{1pt}n} \arrow[r]
  & 0 \arrow[r]
  & \cdots,
\end{tikzcd}

\vspace{-3pt}

\item[\phantom{(i)}] \begin{tikzcd}[column sep=17pt,
row sep=1pt,
column 1/.style={
    nodes={
      text width=4cm,
      align=right
    }
  }]
\mathmakebox[1.3cm][r]{\tau_R^{\geqslant n+1}C^{\bullet}}
  &[-20pt]=
  &[-20pt] \cdots \arrow[r]
  & 0 \arrow[r]
  & \cok d^{\hspace{1pt}n} \arrow[r, "i^{n+1}"]
  & C^{\hspace{0.5pt}n+2} \arrow[r]
  & C^{\hspace{0.5pt}n+3} \arrow[r]
  & \cdots,
\end{tikzcd}

\item[(ii)] $\Der^{\leqslant0}_R(\mathcal{A},\CC)=\bigl\{C^{\bullet}\in\Der(\mathcal{A},\CC)\:|\:\forall\:n\geqslant 1\colon\cok d^{\hspace{1pt}n-1}\stackrel{i^{n}}{\longrightarrow}C^{n+1}\stackrel{p^{n+1}}{\longrightarrow}\cok d^{\hspace{1pt}n}\in\CC\bigr\}$

\vspace{-1pt}

\item[\phantom{(ii)}] $\Der^{\geqslant0}_R(\mathcal{A},\CC)=\bigl\{C^{\bullet}\in\Der(\mathcal{A},\CC)\:|\:\forall\:n\leqslant-1\colon\cok d^{\hspace{1pt}n-1}\stackrel{i^{n}}{\longrightarrow}C^{n+1}\stackrel{p^{n+1}}{\longrightarrow}\cok d^{\hspace{1pt}n}\in\CC\bigr\}$

\vspace{3pt}

\item[(iii)] $\operatorname{RH}^n(C^{\bullet})=\cdots\rightarrow0\rightarrow\cok d^{\hspace{1pt}n-1}\xrightarrow{\pm i^{\hspace{0.5pt}n}}C^{n+1}\xrightarrow{\pm p^{n+1}}\cok d^{\hspace{1pt}n}\rightarrow0\rightarrow\cdots$ with $\cok d^{\hspace{1pt}n-1}$ in degree 0,

\vspace{3pt}

\item[(iv)] $\mathcal{RH}(\mathcal{A},\CC)\cong\hMorc(\mathcal{A})[\mathcal{S}_{\CC}^{-1}]$.\end{myitemize}

\vspace{6pt}

\noindent{}If $(\mathcal{A},\CC)$ has admissible cokernels, then

\begin{myitemize}

\vspace{4pt}

\item[(v)] $\operatorname{RH}^n(C^{\bullet})=\cdots\rightarrow0\rightarrow\cok d^{\hspace{1pt}n-1} \rightarrowtriangle\im d^{\hspace{1pt}n}\rightarrow0\rightarrow\cdots$ with $\cok d^{\hspace{1pt}n-1}$ in degree $0$,

\vspace{4pt}

\item[(vi)] $\mathcal{RH}(\mathcal{A},\CC)\cong\hEpi(\mathcal{A})[\mathcal{S}_{\CC}^{-1}]$.\hfill\diam

\end{myitemize}
\end{rmk}

\section{Classic Examples and Formal Quotients}\label{SEC:CLEX}

In the case of a quasiabelian category the content of Theorem \ref{THM-HEART-2} is due to Schneiders \cite{S} and has been applied to the categories of Banach spaces, Fr\'echet spaces, locally convex (Hausdorff) spaces and bornological vector spaces since the early 2000s \cite{Prosmans, ProsmansSchneiders}. The objects of the left heart can here be interpreted as `formal quotients', an idea promoted already much earlier by Waelbroeck \cite{W, WF, Wbor}, see also the book \cite{WBook}, heuristically addressing the issue that, e.g., in Banach spaces, a vanishing cokernel does not yield surjectivity. Indeed, given a map $f\colon E\rightarrow F$ between Banach spaces its kernel and cokernel in the left heart are given by
\begin{equation*}
\begin{tikzcd}[column sep =20pt, row sep =20pt, ampersand replacement=\&]
\ker\phi(f)\ar[d]\&0\ar[d]\ar[r,hook,]\&f^{-1}(0)\ar[d,"\operatorname{inc}"]\\
\phi(E)\ar[d,swap,"\phi(f)"]\&0\ar[d]\ar[r,hook]\&E\ar[d,"f"]\\
\phi(F)\ar[d]\&0\ar[d]\ar[r,hook]\&F\ar[d,"1"]\\
\cok\phi(f)\&f(E)\ar[r,hook, swap,"\operatorname{inc}"]\&F
\end{tikzcd}
\end{equation*}
where $f(E)$ is endowed with the topology of $E/f^{-1}(0)$ and the inclusion $f(E)\hookrightarrow F$ represents the aforementioned `formal quotient'. Also for other quasiabelian categories the left heart has recently been used and described even more explicitly \cite{Lupini,Braunling}. The generalization in Theorem \ref{THM-HEART-2} to strongly-deflation exact categories with admissible kernels allowed for the treatment of, e.g., the non-quasiabelian category of LB-spaces when endowed with its maximal deflation-exact structure. There we get exactly the same picture as above only that $f^{-1}(0)$ in general carries a topology that is finer than the subspace topology.

\smallskip

While examples for the dual situation have been mentioned, e.g.\ \cite[Ex 9.8\hspace{1pt}--\hspace{1pt}9.10]{HKRW} and \cite[Section 3]{LW}, the latter has so far not been discussed in much detail. In particular, an interpretation of the epimorphism category seems unknown. We will fill this gap in the next chapter focusing mostly on the prototype example of the category $\COM$ of complete locally convex spaces.

\smallskip

In Sections \ref{SEC:COM}\,--\,\ref{SEC:DN} we will indeed discuss several new examples arising as subcategories of the category $\LCS$ of all locally convex spaces and to do so we will frequently use results established in \cite{Prosmans, HSW, KW, SiegThesis}. Recall in particular that in a preabelian category we call a morphism $f\colon E\rightarrow F$ \emph{strict}, if its \emph{parallel} $\bar{f}\colon\coim(f)\rightarrow\im(f)$ is an isomorphism, and that $f$ is a kernel [cokernel] iff it is a strict monomorphism [epimorphism]. Recall moreover that on the quasiabelian categories $\HDLCS\subseteq\LCS$ one naturally considers the maximal exact structures as in Theorem \ref{THMDEFMAX} which consist of the \emph{topologically exact sequences}, i.e., sequences
$$
0\longrightarrow E\xrightarrow{\hspace{2pt}f\hspace{2pt}} F\xrightarrow{\hspace{2pt}g\hspace{2pt}} G\longrightarrow 0
$$
in $\LCS$ or $\HDLCS$ such that the sequence is short exact when considered as a sequence of vector spaces and additionally $f$ and $g$ are open onto their ranges, see e.g.\ \cite[Dfn on p.~309]{MV}. We write $\EE_{\operatorname{top},\HDLCS}\subseteq\EE_{\operatorname{top},\LCS}$ for these two structures.







\section{Complete LcHs}\label{SEC:COM}

We denote by $\COM$ the full subcategory of $\HDLCS$ consisting of the  complete lcHs, see \cite[Thm 3.3]{HSW} and \cite[Section 4.1]{Prosmans}. The latter is well known to be preabelian: kernels compute as in $\HDLCS$ and the cokernel of $f\colon X\rightarrow Y$ is the composition $Y\stackrel{\hspace{-3pt}\scriptscriptstyle q}{\rightarrow}Y/f(X)\stackrel{\hspace{-3pt}\scriptscriptstyle i}{\rightarrow}(Y/\overline{f(X)})^{\raisebox{2pt}{$\scriptstyle\wedge$}}$ of quotient  map and embedding into the completion. We thus may apply Theorem \ref{THMDEFMAX} and denote by $\II_{\text{max},\,\COM}$ and $\DD_{\text{max},\,\COM}$ its maximal inflation- and deflation-exact structures, respectively. 

\begin{thm}\cite[Ex 9.8]{HKRW} The category $(\COM,\II_{\operatorname{max},\,\COM})$ is strongly inflation-exact and has admissible cokernels, the inflation exact structure consists of all kernel-cokernel pairs in $\COM$. The embedding $(\COM,\II_{\operatorname{max},\,\COM})\rightarrow \mathcal{RH}(\COM,\II_{\text{max},\,\COM})$ lifts to an equivalence of bounded derived categories.\hfill\diam{}
\end{thm}

\begin{rmk}\label{COMRMK} The right heart can be described as
$$
\mathcal{RH}(\COM,\II_{\text{max},\,\COM})\cong\hEpi(\COM)[\mathcal{S}_{\II_{\text{max},\,\COM}}^{-1}]
$$
and the embedding $\phi$ of $\COM$ into the right heart sends an object $E\in\COM$ onto the epimorphism $E\epi0$. Using Remark \ref{INF-RMK} and proceeding analogously to the proof of Theorem \ref{KER-COK-PROP} we find that given an $f\colon E\rightarrow F$ in $\COM$, its kernel and cokernel in the right heart look as follows
\begin{equation}\label{DIAG-NEW}
\begin{tikzcd}[column sep =20pt, row sep =20pt, ampersand replacement=\&]
\ker\phi(f)\ar[d]\&E\ar[d,swap,"1"]\ar[r,-open triangle 60,"f|^{\overline{f(E)}}"]\&\overline{f(E)}\ar[d,"0"]\\
\phi(E)\ar[d,swap,"\phi(f)"]\&E\ar[d,swap,"f"]\ar[r,-open triangle 60]\&0\ar[d,"0"]\\
\phi(F)\ar[d]\&F\ar[d,swap,"iq"]\ar[r,-open triangle 60]\&0\ar[d,"0"]\\
\cok\phi(f)\&(F/\overline{f(E)})^{\raisebox{2pt}{$\scriptstyle\wedge$}}\ar[r,-open triangle 60]\&0.
\end{tikzcd}
\end{equation}
The map $f$ is an isomorphism in $\COM$ iff it is an isomorphism in the right heart iff the kernel and cokernel of $f$ in the heart are zero: Indeed, the kernel being zero encodes that $f$ has a closed range and is a topological embedding. If this holds, the cokernel being zero characterizes surjectivity.

\smallskip

Recall that the left heart for lcs, lcHs, Banach spaces, Fr\'echet spaces, LB-spaces, etc.\ adds formal quotients $f(E)\hookrightarrow F$ that serve as the `correct cokernels', whereas the kernels remain unchanged. Now, the cokernels remain unchanged (with the disadvantage that they kill not just $f(E)$ but its closure) and we add what one might call \emph{formal co-quotients} $E\epi\overline{f(E)}$, as can be seen in the top row of \eqref{DIAG-NEW}. These serve as `new kernels' and allow for the recovery of the information of what has been killed besides $f(E)$ in $(F/\overline{f(E)})^{\wedge}$. From a different perspective, if $f$ is the continuous inclusion of a linear subspace and we drop it from notation, then $E\epi\overline{E}$ (with closure taken in $F$) is of course injective and thus may be regarded as a formal quotient in the classic sense again. Notice however that there will be objects in the right heart that are not monomorphisms.
\end{rmk}

We mention that the above holds verbatim for the categories $\mathsf{PLS}$, $\mathsf{PLS}_{\mathsf{w}}$ and $\mathsf{PLN}$ as was noted in \cite[Cor 3.6]{LW} but that there is a misprint in \cite[Rmk 3.7]{LW}, where in the 2nd line the formula for the right heart must be $(\operatorname{hEpi}\mathcal{A})[\{\operatorname{pulations}\}^{-1}]$.

\begin{prop} The category $(\COM,\II_{\operatorname{max},\,\COM})$  has enough injective objects.
\end{prop}
\begin{proof} For every lcHs $E$ there is an embedding into an $\LCS$-injective
$$
i\colon E\rightarrow I\,:=\hspace{-8pt}\Bigprod{U\in\mathcal{U}_0(E)}{}\hspace{-7pt}\ell^{\infty}(U^{\circ}),
$$
see \cite[Thm 2.2.1 + text before it]{JochensBuch} or \cite[Prop 2.1.12 + Cor 2.1.9]{Prosmans}. The product is complete and if $E$ is complete then the range of $i$ is automatically closed, which means that $i$ is a kernel. Consequently, $i$ is an $\II_{\operatorname{max},\,\COM}$-inflation. Finally, $I$ is $(\COM,\II_{\operatorname{max},\,\COM})$-injective: Given $f\colon A\rightarrow I$ and an $\II_{\operatorname{max},\,\COM}$-inflation $g\colon A\rightarrowtail B$, then $g$ is also an $\LCS$-inflation by \cite[Prop 4.1.10(i)]{Prosmans} and thus $f$ lifts along $g$ in $\LCS$ and consequently in $\COM$.
\end{proof}

Let us now turn to the deflation exact structure.

\begin{thm}\label{COM-THM} The category $(\COM,\DD_{\operatorname{max},\,\COM})$ is strongly deflation-exact and a fully exact subcategory of the quasiabelian category $\HDLCS$. The embedding $(\COM,\DD_{\operatorname{max},\,\COM})\rightarrow \mathcal{LH}(\COM,\DD_{\operatorname{max},\,\COM})$ lifts to an equivalence of bounded derived categories and the left heart is given by $\hMork(\COM)[\mathcal{S}_{\DD_{\operatorname{max},\,\COM}}^{-1}]$ as described in Section \ref{SEC-LRH}.
\end{thm}
\begin{proof} Theorem \ref{THMDEFMAX} implies that  $\DD_{\operatorname{max},\,\COM}$ is strongly deflation exact. Using \cite[table on p.\ 2114]{DS12} we get that $\COM\subseteq\HDLCS$ is extension closed and thus the class $\EE_{\operatorname{top},\COM}=\EE_{\operatorname{top},\HDLCS}\cap\COM$ of the topological short exact sequences with all three terms in $\COM$ is an exact structure and
$$
\EE_{\operatorname{top},\COM}\subseteq\DD_{\operatorname{max},\,\COM}.
$$
Given $E\stackrel{\scriptscriptstyle f}{\rightarrow}F\stackrel{\scriptscriptstyle g}{\rightarrow}G$ in $\DD_{\operatorname{max},\,\COM}$, we have by definition that $g$ is a semistable cokernel. The cokernel property in $\COM$ yields openness onto the range. Using \cite[Prop 2.2.3]{SiegThesis} we get that $g$ is surjective. The map $f=\ker g$ is then up to isomorphism the inclusion  $(g^{-1}(0),\tau_F)\rightarrow F$. Hence $(f,g)\in\EE_{\operatorname{top},\COM}$. The remaining statements (and more!) follow from Theorems \ref{THM-LH-PHI}, \ref{NEW-LH} and \ref{KER-COK-PROP}.
\end{proof}

Using the deflation-exact structure appears to be much more natural, since $\DD_{\operatorname{max},\,\COM}$ is even two-sided exact and yields an exact embedding into $\HDLCS$. Also, the embedding into its left heart preserves kernels while cokernels compute as in Theorem \ref{KER-COK-PROP}. Given two complete spaces $E$, $F$ such that $E\subseteq F$ holds  as linear spaces and the inclusion map is continuous, their quotient in the left heart will be $0\rightarrow E\hookrightarrow F$ which is similar to the classical `formal quotient'. Compared to using the inflation exact structure, one has to however work with 3-term complexes in general. Moreover it is, to the authors' best knowledge, unknown if the unrestricted category $(\COM,\DD_{\operatorname{max},\,\COM})$ has a locally small bounded derived category.





\section{Hypo-Complete LcHs}\label{SEC:HCOM}

In this section we consider the three full subcategories of $\HDLCS$ formed by the quasi-complete, the sequentially complete and the locally (or Mackey) complete lcHs, respectively. All three have important applications in different areas of mathematics: quasi-complete spaces play an important role in Lie theory \cite{Ga_cqc}; sequentially complete spaces are more intuitive than complete spaces and often sufficient in applications, e.g., in ergodic theory, functional calculus or evolution equations \cite{ABR, Stoian, AK02, KMS21, XL}; locally complete spaces occur, e.g., in convenient analysis, vector-valued holomorphy and again Lie theory \cite{KM97, Teichmann, KK21, GlNe}.

\begin{dfn}\cite[Chapters 3.2 and 10.2]{Jarchow} A lcHs $X$ is called\vspace{2pt}
\begin{myitemize}

\item[(i)] \emph{quasi-complete} if every closed and bounded set in $X$ is complete,\vspace{3pt}

\item[(ii)] \emph{sequentially complete} if every Cauchy sequence in $X$ converges,\vspace{3pt}

\item[(iii)] \emph{locally} (or \emph{Mackey}) \emph{complete} if every closed, bounded and absolutely convex set $B\subseteq X$ is a Banach disk, i.e., $\spann B$ endowed with its Minkowski functional $p_B$ is a Banach space.\diam{}
\end{myitemize}
\end{dfn}

Notice that local completeness has several other equivalent characterizations \cite[Ch 10.2]{Jarchow} that better explain its name. The following properties/constructions are well-known for all three types of completeness properties.

\begin{prop}\label{HYPOPROP} Let `hypo-' stand either for `quasi-', `sequentially' or `locally'.\vspace{2pt}

\begin{myitemize}

\item[(i)] For every lcHs $E$ there exists a hypo-complete space $\qcomp{E}$ together with a map $i\colon E\rightarrow\qcomp{E}$ which is injective, continuous, linear, open onto its range and has a dense range such that for any hypo complete space $F$ and any continuous linear map $u\colon E\rightarrow F$ there is a unique continuous map $v\colon\qcomp{E}\rightarrow F$ such that\begin{equation*}
\begin{tikzcd}
\qcomp{E}\arrow{r}{v} & F \\ 
E \arrow{u}{i}\arrow{ur}[swap]{u} & 
\end{tikzcd}
\end{equation*}
commutes. We call $\qcomp{E}$ the \emph{hypo-completion} of $E$.

\vspace{3pt}

\item[(ii)] If $f\colon E\rightarrow F$ is a morphism in $\HDLCS$, then there is a unique extension $\qcomp{f}\colon\qcomp{E}\rightarrow\qcomp{F}$ of $f$ to the hypo-completions, i.e., it holds $\qcomp{f}i_E=i_Ff$.

\vspace{3pt}

\item[(iii)] If $E$ is a hypo-complete lcHs and $F\subseteq E$ is a closed subspace that we endow with the subspace topology, then $F$ is again hypo-complete.

\vspace{3pt}

\item[(iv)] The topological product of finitely many hypo-complete lcHs is hypo-complete.

\vspace{3pt}

\item[(v)] There exists a hypo-complete lcHs that is not complete.
\vspace{3pt}

\item[(vi)] There exists a hypo-complete lcHs $E$ and a closed subspace $F\subseteq E$ such that $E/F$ is not hypo-complete.

\end{myitemize}
\end{prop}
\begin{proof}(i) The hypo-completion can be constructed via
$$
\qcomp{E}\hspace{4pt}=\hspace{-4pt}\mathop{\textstyle\bigcap}_{\stackrel{E\subseteq F\subseteq \widehat{E}}{\stackrel{F \text{ hypo-}}{\scriptscriptstyle\text{complete}}}}\hspace{-3pt}F
$$ 
where $\widehat{E}$ is the usual completion. The map $i$ is the corestriction of the inclusion map $E\subseteq\widehat{E}$ and the universal property can be derived from that of $\widehat{E}$. For the three cases this can be found in \cite[\S\,23.1]{KI}, \cite[Dfn 5.1.21]{BPC} and \cite[p.~14]{ORS} (where the intersection is taken over subspaces of the quasi-completion which yields however the same space).

\smallskip

(ii) Consider $\widehat{f}\colon\widehat{E}\rightarrow\widehat{F}$ and define $\qcomp{f}$ as its (co)restriction. That the latter is well-defined can be found in \cite[\S\,23.1(4)]{KI} for the quasi-completion and in \cite[Lem 5.1.23\,--\,Prop 5.1.25]{BPC} for the local completion. For the sequential completion one can argue analogously.

\smallskip

(iii) For quasi-complete or sequentially complete $E$ this is clear, for locally complete it can be found in \cite[Prop 5.1.13]{BPC}.

\smallskip

(iv) See \cite[Prop 4.4.3]{Jarchow} for quasi- and sequential completeness. For local completeness the proof is straightforward.

\smallskip

(v) In view of the implications
\begin{center}
complete\;$\Longrightarrow$\;quasi-complete\;$\Longrightarrow$\;sequentially complete\;$\Longrightarrow$\;locally complete
\end{center}
it suffices to find a quasi-complete lcHs that fails to be complete. Such an example is given in \cite[\S\,23.1]{KI} or, with more details, in \cite[Chapter III, \S\,1, no.\ 6, text after Dfn 6]{EVT}.

\smallskip

(vi) Let $G$ be any lcHs that is not locally complete, e.g, an incomplete normed space. Then by \cite{Dier75} there exists a complete lcHs $E$ and a closed subspace $F\subseteq E$ such that $G\cong{}E/F$.
\end{proof}

For the rest of this section hypo-complete will stand for either one of the properties quasi-complete, sequentially complete or locally complete; we denote by $\HCOM$ the full subcategory of $\HDLCS$ formed by the hypo-complete lcHs. By $\overbracket[0.6pt][1.5pt]{\hspace{2.5pt}\cdot\hspace{2.5pt}}$ or $\hc\colon\HDLCS\rightarrow\HCOM$ we denote the hypo-completion functor that sends a space $E$ to its hypo-completion $\overbracket[0.6pt][1.5pt]{E\hspace{1pt}}$ and a morphism $f\colon E\rightarrow F$ to its extension $\overbracket[0.6pt][1.5pt]{f\hspace{1pt}}\colon\overbracket[0.6pt][1.5pt]{E\hspace{1pt}}\rightarrow\overbracket[0.6pt][1.5pt]{F\hspace{1pt}}$.

\begin{lem}\label{LEM-1} The functor $\hc$ is left adjoint to the inclusion functor $\inc$. 
\end{lem}
\begin{proof}For $E\in\HDLCS$ and $F\in\HCOM$ the restriction $f\mapsto f|_E$ yields an isomorphism
$$
\Hom_{\hspace{1pt}\HCOM}(\hc(E),F)\stackrel{\hspace{-2pt}\sim}{\longrightarrow}\Hom_{\hspace{1pt}\HDLCS}(E,\inc(F))
$$
of vector spaces with inverse $f\mapsto{\overbracket[0.6pt][1.5pt]{f\hspace{1pt}}}$ and the isomorphism is natural in $E$ and $F$.
\end{proof}

From Proposition \ref{HYPOPROP} we get that $\HCOM$ is an additive category whose finite biproducts are precisely those of $\HDLCS$. In particular, $\HCOM\subseteq\HDLCS$ is a full additive subcategory which is by Lemma \ref{LEM-1} reflective. One may consider $\hc\colon\LCS\rightarrow\HCOM$ also on the category of all possibly non-Hausdorff lcs; Prosmans did this for the (normal) completion functor. The latter then however forms the Hausdorff-completion, i.e., it first quotients out the closure of zero, and then forms the completion. If $E$ is not Hausdorff, the canonical maps $E\rightarrow\widehat{E}$ and $E\rightarrow\overbracket[0.6pt][1.5pt]{E\hspace{1pt}}$ will not be injective.

\begin{prop}\label{PROP-1} The category $\HCOM$ is preabelian; for $f\colon E\rightarrow F$ in $\HCOM$ we have:\vspace{3pt}
\begin{compactitem}
\item[(i)] The kernel of $f$ is the inclusion map $(f^{-1}(0),\tau_{E|f^{-1}(0)})\rightarrow E$.

\vspace{3pt}

\item[(ii)] The cokernel of $f$ is the composition $F\xrightarrow{\hspace{1pt}\scriptscriptstyle q\hspace{1pt}}F/\overline{f(E)}\xrightarrow{\hspace{1pt}\scriptscriptstyle i\hspace{1pt}}\overbracket[0.6pt][1.5pt]{F/\overline{f(E)}\hspace{2pt}}$.

\vspace{4pt}

\item[(iii)] The image of $f$ is the inclusion map $(\overline{f(E)}^{F},\tau_{F|\overline{f(E)}^F})\rightarrow F$. \vspace{0pt}

\item[(iv)] The coimage of $f$ is the composition $E\xrightarrow{\hspace{1pt}\scriptscriptstyle q\hspace{1pt}}E/f^{-1}(0)\xrightarrow{\hspace{1pt}\scriptscriptstyle i\hspace{1pt}}\overbracket[0.6pt][1.5pt]{E/f^{-1}(0)\hspace{2pt}}$.
\end{compactitem}
\end{prop}
\begin{proof} (i) As we work with Hausdorff spaces, $f^{-1}(0)\subseteq E$ is closed and thus by Proposition \ref{HYPOPROP} again hypo-complete. Since the inclusion $(f^{-1}(0),\tau_{E|f^{-1}(0)})\rightarrow E$ is a kernel of $f$ in $\HDLCS$, it is in $\HCOM$, too.

\smallskip

\noindent(ii) We have to check that $iq$ satisfies the universal property of the cokernel. For this let $g\colon F\rightarrow G$ be a morphism in $\HCOM$ such that $gf=0$. Since $q$ is the cokernel of $f$ in $\HDLCS$ we get a unique $h\colon F/\overline{f(E)}\rightarrow G$ with $hq=g$. Applying Proposition \ref{HYPOPROP}(i) yields a unique linear and continuous map
$$
k\colon \overbracket[0.6pt][1.5pt]{F/\overline{f(E)}\hspace{2pt}}\rightarrow G
$$
with $ki=h$. It follows $kiq = g$. If $k'iq = g$ holds too, then by surjectivity of $q$ we get $k'i=ki=h$ which implies $k'=k$.

\smallskip

\noindent(iii) and (iv) follow from (i) and (ii).
\end{proof}

We note that the first two statements in Proposition \ref{PROP-1} could also be derived by the abstract categorical arguments in \cite[Proposition 4.5.15]{Riehl} employing Lemma \ref{LEM-1}.

\begin{cor}\label{COR-1} Let $f\colon E\rightarrow F$ be a morphism in $\HCOM$. Then\vspace{2pt}
\begin{compactitem}
\item[(i)] $f$ is a monomorphism iff $f$ is injective,\vspace{2pt}

\item[(ii)] $f$ is an epimorphism iff $f$ has a dense range,\vspace{2pt}

\item[(iii)] its parallel $\bar{f}\colon\coim(f)\rightarrow\im(f)$ is always an epimorphism.\end{compactitem}
\end{cor}
\begin{proof} (i) and (ii) As $\HCOM$ is preabelian, $f$ is monic iff $\ker f=0$ and this is precisely the case if $f$ is injective. Dually, $f$ is epic iff $\cok f=0$ and this holds precisely if $f(E)\subseteq F$ is dense. 

\smallskip

(iii) Using Proposition \ref{PROP-1} we see that $\bar{f}\colon\overbracket[0.6pt][1.5pt]{E/f^{-1}(0)\hspace{2pt}}\longrightarrow \overline{f(E)}$ arises by applying the hypo-completion functor to the map $E/f^{-1}(0)\rightarrow \overline{f(E)}$. As $\overline{f(E)}$ carries the topology induced by $F$, the latter map has a dense range and thus so does $\overline{f}$. By (ii) it is thus an epimorphism.
\end{proof}

Corollary \ref{COR-1}(iii) means that $\HCOM$ is right semiabelian \cite{HSW, KW}, while the following example shows that $\HCOM$ is not left semiabelian.

\begin{ex}\label{EX-1} There is a morphism $f\colon E\rightarrow F$ in $\HCOM$ such that $\bar{f}$ is not a monomorphism.
\end{ex}
\begin{proof} We proceed exactly as in \cite[Prop 3.1.6]{SiegThesis}: Let $E$ be a hypo-complete space and $U\subseteq E$ a closed subspace such that $E/U$ is not hypo-complete, cf.~Proposition \ref{HYPOPROP}, and define $F:=\overbracket[0.6pt][1.5pt]{E/U\hspace{2pt}}$. Let $g\colon E\rightarrow F$ be the natural map and select $x_0\in F\backslash g(E)$.  Define
$$
f\colon E\oplus\KK\rightarrow F,\;(e,\lambda)\mapsto g(e)-\lambda{}x_0.
$$
Then we get $\coim f = \overbracket[0.6pt][1.5pt]{\zfrac{E\oplus\KK}{\ker f}\hspace{2pt}}= \overbracket[0.6pt][1.5pt]{\zfrac{E\oplus\KK}{U\oplus\{0\}}\hspace{2pt}} = F\oplus\KK $ and $\im f = \overline{f(E)}=F$, and $\bar{f}\colon F\oplus\KK\rightarrow F$ is given by $(x,\lambda)\mapsto x-\lambda x_0$ which is not injective.
\end{proof}

\begin{prop}\label{PROP-2} A morphism $f\colon E\rightarrow F$ in $\HCOM$ is strict iff $f$ is open onto its range and $\overbracket[0.6pt][1.5pt]{f(E)\hspace{1pt}}=\overline{f(E)}$ holds in the sense that $\overbracket[0.6pt][1.5pt]{j\hspace{1pt}}$ is an isomorphism, where $j\colon f(E)\rightarrow\overline{f(E)}$ denotes the inclusion. 
\end{prop}
\begin{proof} We consider the following diagram in the category $\HDLCS$, where $f(E)$ and $\overline{f(E)}$ are endowed with the subspace topology of $F$, $\bar{f}^{\hspace{2pt}\raisebox{2.5pt}{$\scriptscriptstyle\HCOM$}}$ is the parallel in $\HCOM$ and $\bar{f}^{\hspace{2pt}\raisebox{2.5pt}{$\scriptscriptstyle\LCS$}}$ is the parallel in $\LCS$ (which is bijective and continuous, but in general not open); the maps $i$ and $j$ are inclusions of topological subspaces:
\begin{equation*}
\begin{tikzcd}[row sep =27pt, column sep = 34pt]
E/f^{-1}(0)\arrow{r}{\bar{f}^{\hspace{2pt}\raisebox{1pt}{$\scriptscriptstyle\LCS$}}}\arrow{d}[swap]{i\hspace{1pt}} & f(E) \arrow{d}{j}  \\
\overbracket[0.6pt][1.5pt]{E/f^{-1}(0)\hspace{1pt}} \arrow{r}[swap]{\bar{f}^{\hspace{2pt}\raisebox{2.5pt}{$\scriptscriptstyle\HCOM$}}} & \overline{f(E)}.
\end{tikzcd}
\end{equation*}
Observe that since $\overline{f(E)}$ is hypo-complete, we get indeed $\overbracket[0.6pt][1.5pt]{j\hspace{2pt}}\colon\overbracket[0.6pt][1.5pt]{f(E)}\rightarrow\overline{f(E)}$ by extending $j$ to the hypo-completion of its domain, see Proposition \ref{HYPOPROP}.

\smallskip

\textquotedblleft{}$\Longrightarrow$\textquotedblright{} Assume now that $f$ is strict in $\HCOM$. Then $\bar{f}^{\hspace{2pt}\raisebox{2.5pt}{$\scriptscriptstyle\HCOM$}}$ is an isomorphism by definition and it follows that its restriction $\bar{f}^{\hspace{2pt}\raisebox{2.5pt}{$\scriptscriptstyle\LCS$}}$ is an isomorphism, too. If we identify $E/f^{-1}(0)\equiv{}f(E)$ via $\bar{f}^{\hspace{2pt}\raisebox{2.5pt}{$\scriptscriptstyle\LCS$}}$, we get that $\overbracket[0.6pt][1.5pt]{j\hspace{2pt}}=\bar{f}^{\hspace{2pt}\raisebox{2.5pt}{$\scriptscriptstyle\HCOM$}}$ is an isomorphism. Moreover, $f$ is then strict in $\LCS$ and thus open onto its range by \cite[Cor 2.1.9]{Prosmans}.

\smallskip

\textquotedblleft{}$\Longleftarrow$\textquotedblright{} Let $f$ be open onto its range. Then $\bar{f}^{\hspace{2pt}\raisebox{2.5pt}{$\scriptscriptstyle\LCS$}}$ is an isomorphism and we may identify $E/f^{-1}(0)\cong{}f(E)$. Now, $\bar{f}^{\hspace{2pt}\raisebox{2.5pt}{$\scriptscriptstyle\HCOM$}}=\overbracket[0.6pt][1.5pt]{j\hspace{2pt}}$ is an isomorphism by assumption.
\end{proof}

\begin{cor}\label{COR-2} If a morphism $f\colon E\rightarrow F$ in $\HCOM$ is open onto its range and its range is closed, then $f$ is strict in $\HCOM$.
\end{cor}
\begin{proof} If $f(E)\subseteq F$ is closed, then $f(E)$ is hypo-complete and coincides with its hypo-completion, thus $\overbracket[0.6pt][1.5pt]{j\hspace{2pt}}=\id_{f(E)}$ is an isomorphism.
\end{proof}

\begin{rmk}\label{RMK-1} In the category $\COM$ of complete lcHs a morphism $f\colon E\rightarrow F$ is strict iff it is open onto its range by \cite[Prop 4.1.9 together with Cor 2.1.9]{Prosmans}\,---\,no further condition is needed. The reason is that taking a closure in a complete superspace and forming the completion are the same thing. In $\HCOM$ this is in general \emph{not} the case. In particular, it might happen that $\overbracket[0.6pt][1.5pt]{f(E)\hspace{1pt}}\subset\overline{f(E)}$ is a proper subspace: Take for example a hypo-complete space $E$ which is not complete as in Proposition \ref{HYPOPROP} and let $f\colon E\rightarrow \widehat{E}$ be the inclusion. Then $f(E)\cong E$ which is hypo-complete, hence $\overbracket[0.6pt][1.5pt]{f(E)\hspace{1pt}}=f(E)$ holds. On the other hand the closure, taken in $\widehat{E}$, equals $\overline{f(E)}=\widehat{E}$ and is strictly larger than $E$ by construction.\diam{}
\end{rmk}

Our next aim is to identify the kernels\,($\equiv$\,strict monomorphisms) and the cokernels\,($\equiv$\,strict epimorphisms) in $\HCOM$.

\begin{prop}\label{PROP-3} Let $f\colon E\rightarrow F$ be a morphism in $\HCOM$. Then\vspace{2pt}
\begin{myitemize}
\item[(i)] $f$ is a kernel iff $f$ is injective, open onto its range and its range is closed.\vspace{2pt}

\item[(ii)] $f$ is a cokernel iff $f$ is open onto its range and $F=\qcomp{f(E)}$ holds.\vspace{2pt}
\end{myitemize}
\end{prop}
\begin{proof} (i) \textquotedblleft{}$\Longrightarrow$\textquotedblright{} Follows from Proposition \ref{HYPOPROP}.

\smallskip

\textquotedblleft{}$\Longleftarrow$\textquotedblright{} If $f$ is injective, open onto its range and has a closed range, then it is monic by Corollary \ref{COR-1} and strict by Corollary \ref{COR-2}. 
 
\smallskip
 
(ii) \textquotedblleft{}$\Longrightarrow$\textquotedblright{} Follows from Proposition \ref{HYPOPROP}.

\smallskip

\textquotedblleft{}$\Longleftarrow$\textquotedblright{} Let $f$ be open onto its range and assume $F=\qcomp{f(E)}$. We form the kernel of $f$, denote it by $i\colon f^{-1}(0)\rightarrow E$ and claim that $f=\cok i$ holds in $\HCOM$. We have $fi=0$ by definition. Let $hi=0$ for some $h\colon E\rightarrow G$. We form the parallel $\bar{f}^{\hspace{2pt}\raisebox{1pt}{$\scriptscriptstyle\LCS$}}$ of $f$ in $\LCS$ and obtain in $\LCS$ the factorization 
\begin{equation*}
\begin{tikzcd}
E\arrow{r}{f}\arrow{d}[swap]{q} & F \\
E/f^{-1}(0) \arrow{r}[swap]{\bar{f}^{\hspace{2pt}\raisebox{1pt}{$\scriptscriptstyle\LCS$}}} & f(E)\arrow{u}[swap]{j}
\end{tikzcd}
\end{equation*}
where $q$ is the quotient map, $j$ is the inclusion and $\bar{f}^{\hspace{2pt}\raisebox{1pt}{$\scriptscriptstyle\LCS$}}$ is an isomorphism. As $q$ is the cokernel of $i$ in $\LCS$, we get a unique $\alpha\colon E/f^{-1}(0)\rightarrow G$ with $\alpha\hspace{1pt}q = h$ and may consider the composition
$$
\alpha\hspace{1pt}(\bar{f}^{\hspace{2pt}\raisebox{1pt}{$\scriptscriptstyle\LCS$}})^{-1}\colon f(E)\longrightarrow G
$$
which is a linear and continuous map. By applying the hypo-completion functor to $\alpha\hspace{1pt}(\bar{f}^{\hspace{2pt}\raisebox{1pt}{$\scriptscriptstyle\LCS$}})^{-1}$, we extend the latter to a map $\beta\colon F\rightarrow G$ and get $\beta f = \beta\hspace{1pt}j\hspace{1pt}\bar{f}^{\hspace{2pt}\raisebox{1pt}{$\scriptscriptstyle\LCS$}}\hspace{1pt}q = \alpha\hspace{1pt}(\bar{f}^{\hspace{2pt}\raisebox{1pt}{$\scriptscriptstyle\LCS$}})^{-1}\hspace{1pt}\bar{f}^{\hspace{2pt}\raisebox{1pt}{$\scriptscriptstyle\LCS$}}\hspace{1pt}q= \alpha\hspace{1pt}q = h$. Uniqueness of $\beta$ follows since $f$ has a dense range.
\end{proof}

Now we get immediately the following description of $\CC_{\operatorname{all},\,\HCOM}$.

\begin{prop}\label{PROP-KK} A sequence $E\xrightarrow{\hspace{2pt}\scriptscriptstyle f\hspace{2pt}}F\xrightarrow{\hspace{2pt}\scriptscriptstyle g\hspace{2pt}}G$ in $\HCOM$ is a kernel-cokernel pair iff $f$ is injective, open onto its range, $f(E)=g^{-1}(0)$ holds as vector spaces and $g$ is open onto its range and it holds $G=\qcomp{g(F)}$.
\end{prop}
\begin{proof}\textquotedblleft{}$\Longrightarrow$\textquotedblright{} It follows from Proposition \ref{PROP-3} that $f$ is injective, open onto its range and that $f(E)\subseteq F$ is closed. Moreover it follows from the same proposition that $g$ is open onto its range and that $G=\qcomp{g(F)}$ holds. From $f=\ker g$ it follows that $f$ is, up to isomorphism, the inclusion of the subspace $g^{-1}(0)$ into $F$, which means $f(E)=g^{-1}(0)$.

\smallskip

\textquotedblleft{}$\Longleftarrow$\textquotedblright{} The conditions that apply to $f$ resp.~$g$ separately yield that $f$ is a kernel and $g$ is a cokernel. It suffices therefore to show that $f=\ker g$, which follows from $f(E)=g^{-1}(0)$.
\end{proof}

In $\HCOM$ cokernels are in general not surjective and thus kernel-cokernel pairs need not to be algebraically exact: Take $F\subseteq E$ as in Proposition \ref{HYPOPROP}(vi) and consider the sequence 
$$
F\hookrightarrow E\epi\overbracket[0.6pt][1.5pt]{E/\overline{F}\hspace{1pt}}.
$$

\begin{prop}\label{PROP-4} Consider the functors $\inc\colon\HCOM\rightarrow\HDLCS$ and $\hc\colon\HDLCS\rightarrow\HCOM$.\vspace{2pt}
\begin{myitemize}

\item[(i)] The functor $\inc$ is kernel preserving; more precisely: If $f\colon F\rightarrow G$ is a morphism in $\HCOM$ and $g\colon E\rightarrow F$ is its kernel in $\HCOM$, then $g=\ker f$ holds in $\HDLCS$.

\vspace{3pt}

\item[(ii)] The functor $\hc$ is cokernel preserving; more precisely: If $f\colon E\rightarrow F$ is a morphism in $\HDLCS$ and $g\colon F\rightarrow G$ is its cokernel in $\HDLCS$, then $\overbracket[0.6pt][1.5pt]{g\hspace{1pt}}=\cok \overbracket[0.6pt][1.5pt]{f\hspace{1pt}}$ holds in $\HCOM$.

\end{myitemize}
\end{prop}

\begin{proof}As $\hc\colon\HDLCS\rightarrow\HCOM$ is a left adjoint and $\inc\colon\HCOM\rightarrow\HDLCS$ is a right adjoint by Lemma \ref{LEM-1}, it follows from \cite[Thm 4.5.3]{Riehl} that $\hc$ preserves cokernels and from \cite[Thm 4.5.2]{Riehl} that $\inc$ preserves kernels.
\end{proof}

\begin{rmk}\label{RMK-2} We note the following non-properties of the functors $\hc$ and $\inc$.\vspace{3pt}

\begin{myitemize}

\item[(i)] The functor $\hc$ does \emph{not} preserve kernels: Consider the restriction map
$$
r\colon\mathbb{P}[0,2]\rightarrow\mathbb{P}[0,1],\;p\mapsto p|_{[0,1]}
$$
between spaces of real-valued polynomial functions on the given intervals, both endowed with the supremum norm over that interval. Then $r$ is  linear, continuous and injective, hence the kernel of $r$ is $0\colon\{0\}\rightarrow\mathbb{P}[0,2]$. As we deal with normed spaces, completion and hypo-completion coincide and we get
$$
\overbracket[0.6pt][1.5pt]{\hspace{1pt}r\hspace{1pt}}\colon\operatorname{C}[0,2]\rightarrow \operatorname{C}[0,1],\;f\mapsto f|_{[0,1]}
$$
which must be the restriction as well, since the latter is an extension of $r$ and $\overbracket[0.6pt][1.5pt]{\hspace{1pt}r\hspace{1pt}}$ is unique. The kernel of $\overbracket[0.6pt][1.5pt]{\hspace{1pt}r\hspace{1pt}}$ is now however non-zero and thus does not coincide with $\hc(0)=0\colon\{0\}\rightarrow\operatorname{C}[0,2]$. Notice that the latter map is nevertheless a kernel, but not of $\overbracket[0.6pt][1.5pt]{\hspace{1pt}r\hspace{1pt}}$.

\vspace{3pt}

\item[(ii)] The functor $\inc$ does not map kernel-cokernel pairs to kernel-cokernel pairs; consider sequence given before Proposition \ref{PROP-4} and observe that $F\rightarrow E\rightarrow E/\overline{F}\in\EE_{\operatorname{top},\,\HDLCS}$.
\end{myitemize}
\end{rmk}

Notice that the following result would also hold true also if we  take for $i\colon E\rightarrow F$ a morphism in $\LCS$ and read $\hc$ as the Hausdorff hypo-completion.

\begin{prop}\label{PROP-5} Let $f\colon E\rightarrow F$ in $\HDLCS$ be injective and open onto its range. Then $\overbracket[0.6pt][1.5pt]{f\hspace{1pt}}\colon \overbracket[0.6pt][1.5pt]{E\hspace{1pt}}\rightarrow \overbracket[0.6pt][1.5pt]{F\hspace{1pt}}$ is injective and open onto its range, too. 
\end{prop}
\begin{proof} Employing \cite[Prop 4.1.12]{Prosmans} the map $\widehat{f}\colon\widehat{E}\rightarrow\widehat{F}$ is a strict monomorphism in $\COM$; in particular injective and open onto its range. As $\overbracket[0.6pt][1.5pt]{f\hspace{1pt}}$ is the restriction of $\widehat{f}$ the conclusion follows either by directly checking that $\overbracket[0.6pt][1.5pt]{f\hspace{1pt}}$ is open onto its range or by using that the restriction of the inverse of $\widehat{f}\colon \widehat{E}\rightarrow\ran\widehat{f}$ to $\ran\overbracket[0.6pt][1.5pt]{f\hspace{1pt}}$ is continuous and yields an inverse for $\overbracket[0.6pt][1.5pt]{f\hspace{1pt}}\colon\overbracket[0.6pt][1.5pt]{E\hspace{1pt}}\rightarrow\ran\overbracket[0.6pt][1.5pt]{f\hspace{1pt}}$.
\end{proof}

The proofs of the following two results are very similar to \cite[Prop 4.6 and Lem 4.11]{Wegner25}, we include them for the convenience of the reader.

\begin{prop}\label{PROP-6} Let $g\colon F\rightarrow G$ be a cokernel in $\HCOM$. Then $g$ is a semistable cokernel in $\HCOM$ iff $g$ is surjective.
\end{prop}
\begin{proof}\textquotedblleft{}$\Longrightarrow$\textquotedblright{} Follows from \cite[Prop 2.2.3]{SiegThesis}.

\vspace{3pt}

\textquotedblleft{}$\Longleftarrow$\textquotedblright{} A $\HCOM$-cokernel $g$ is open onto its range by Proposition \ref{PROP-3}(ii). By assumption, $g$ is surjective. It is thus a cokernel in $\HDLCS$ by \cite[Cor 2.1.9 and 3.1.5(ii)]{Prosmans}. Let $t\colon T\rightarrow G$ be a morphism in $\HCOM$. In view of Proposition \ref{PROP-1}(i) the $\HCOM$-pullback $(P,p_T,p_F)$ of $g$ along $t$ coincides with the $\HDLCS$-pullback. As $\HDLCS$ is quasiabelian, $p_T\colon P\rightarrow T$ is an $\HDLCS$-cokernel and thus in particular surjective, thus by Proposition \ref{PROP-3}(ii) a $\HCOM$-cokernel.
\end{proof}

\begin{lem}\label{LEM-2} We consider everything inside the category $\HCOM$. Let $E\xrightarrow{\hspace{2pt}\scriptscriptstyle f\hspace{2pt}}F\xrightarrow{\hspace{2pt}\scriptscriptstyle g\hspace{2pt}}G$ be a kernel-cokernel pair such that $f$ is a semistable kernel and $g$ is surjective. Let $t\colon E\rightarrow T$ be an arbitrary morphism. Then $S:=(F\oplus T)/\overline{\ran[f\:\sm{}\hspace{-2pt}t]^{\T}}$ is hypo-complete.
\end{lem}
\begin{proof}\textcircled{1} We form the pushout of $f$ along $t$ in $\HDLCS$. This yields precisely the space $S$ with the natural maps and thus the solid part of the following diagram:
\begin{equation*}
\begin{tikzcd}[column sep =32pt, row sep =28pt]
E\arrow{r}{f}\arrow{d}[swap]{t} \commutes[\text{$\HDLCS$-PO}]{dr}& F \arrow{d}{s_F}\arrow{r}{g} &G\arrow{d}{\id_G}\\
 T \arrow{r}[swap]{s_T} & S\arrow[dashed]{r}[swap]{c}& G.
\end{tikzcd}
\end{equation*}
By Proposition \ref{PROP-3}(i) and \cite[Cor 3.1.5 and 2.1.9]{Prosmans} a morphism in $\HCOM$ is a $\HCOM$-kernel iff it is a $\HDLCS$-kernel. Using this and that $\HDLCS$ is quasiabelian, we conclude that $s_T$ is a kernel in $\HDLCS$. By the pushout property indicated in the diagram there is a unique map $c\colon S\rightarrow G$ such that $g=cs_F$ and $cs_T=0$. By \cite[Thm 10]{RW} this map $c$ is the cokernel of $s_T$ in the category in which we took the pushout, i.e., in $\HDLCS$. From the diagram we see that $c$ is surjective as $g$ was surjective by assumption. Moreover, we get that $(s_T,c)$ is a kernel-cokernel pair in $\HDLCS$.

\smallskip

\textcircled{2} We extend the diagram from part \textcircled{1} by using the embedding $i\colon S\rightarrow \overbracket[0.6pt][1.5pt]{S\hspace{1pt}}$ and get the solid part of the following diagram in which the left rectangle is a pushout square in $\HCOM$:
\begin{equation*}
\begin{tikzcd}[column sep =32pt, row sep =28pt]
E\arrow{r}{f}\arrow{d}[swap]{t} & F \arrow{d}{s_F}\arrow{r}{g} &G\arrow{d}{\id_G}\\
T\arrow{d}[swap]{\id_T} \arrow{r}[swap]{s_T} & S\arrow{r}[swap]{c}\arrow{d}{i}& G\arrow{d}{\id_G}\\
T\arrow{r}[swap]{is_T}& \overbracket[0.6pt][1.5pt]{S\hspace{1pt}}\arrow[dashed]{r}[swap]{\overbracket[0.6pt][1.5pt]{\scriptstyle\hspace{0.5pt}c\hspace{0.75pt}}}&G.
\end{tikzcd}
\end{equation*}
Now we apply the hypo-completion functor to obtain $\overbracket[0.6pt][1.5pt]{\hspace{0.5pt}c\hspace{1pt}}$ which makes the lower right square commutative and must be surjective as $c$ was already surjective by \textcircled{1}. By commutativity we obtain
$$
\overbracket[0.6pt][1.5pt]{\hspace{0.5pt}c\hspace{1pt}}\hspace{2.5pt} (is_T)=0 \;\text{ and } \; \overbracket[0.6pt][1.5pt]{\hspace{0.5pt}c\hspace{1pt}}\hspace{2.5pt}(is_F)=g,
$$
which means that $\overbracket[0.6pt][1.5pt]{\hspace{0.5pt}c\hspace{1pt}}$ is precisely the map that comes from the pushout property of the left rectangle in $\HCOM$.  Invoking \cite[Thm 10]{RW} a second time, but now applying it in $\HCOM$, we get that $\overbracket[0.6pt][1.5pt]{\hspace{0.5pt}c\hspace{1pt}}$ is a $\HCOM$-cokernel of $is_T$. By assumption, $f$ is stable under pushouts in $\HCOM$ which implies that $is_T$ is a $\HCOM$-kernel and thus $(is_T,\overbracket[0.6pt][1.5pt]{\hspace{0.5pt}c\hspace{1pt}}\hspace{1pt})$ is a $\HCOM$-kernel-cokernel pair.

\smallskip 

\textcircled{3} Let us now for a moment forget all topologies in the diagram. Then \textcircled{1} and \textcircled{2} imply that the second and third row, that is the pairs $(s_T,c)$ and $(is_T,\overbracket[0.6pt][1.5pt]{\hspace{0.5pt}c\hspace{1pt}}\hspace{1pt})$, are both short exact sequences of vector spaces. Using the 5-Lemma we get that $i$ is an isomorphism of vector spaces. As $i$ was the inclusion of a subspace, this means that $S=\overbracket[0.6pt][1.5pt]{S\hspace{1pt}}$ holds as claimed.
\end{proof}

\begin{prop}\label{PROP-7} We consider everything inside the category $\HCOM$. Let $E\xrightarrow{\hspace{2pt}\scriptscriptstyle f\hspace{2pt}}F\xrightarrow{\hspace{2pt}\scriptscriptstyle g\hspace{2pt}}G$ be a kernel-cokernel pair and assume that $g$ is surjective. Then the following are equivalent:\vspace{2pt}

\begin{compactitem}

\item[(i)] $f$ is a semistable kernel,\vspace{3pt}

\item[(ii)] For every $t\colon E\rightarrow T$ the quotient $S:=(F\oplus T)/\overline{\ran[f\:\sm{}\hspace{-2pt}t]^{\T}}$ is hypo-complete.

\end{compactitem}
\end{prop}
\begin{proof}(i)\hspace{1pt}$\Longrightarrow$\hspace{1pt}(ii): Lemma \ref{LEM-2}.

\smallskip

(ii)\hspace{1pt}$\Longrightarrow$\hspace{1pt}(i): If $t\colon E\rightarrow T$ in $\HCOM$ is given, then by (ii) the pushout of $f$ along $t$ in $\HCOM$ coincides with the pushout in $\HDLCS$. From Proposition \ref{PROP-4}(i) we get that $f$ is a kernel in $\HDLCS$. As the latter category is quasiabelian, $s_T\colon T\rightarrow S$ is a kernel in $\HDLCS$ and belongs to $\HCOM$. As noted in part \textcircled{1} of the proof of Lemma \ref{LEM-2} it follows that $s_T$ is then a kernel in $\HCOM$ as well.
\end{proof}

Below we use the notation of Section \ref{SEC:PRELIM} and what we recalled at the end of Section \ref{SEC:CLEX} on topologically exact sequences. Notice that while $\COM\subseteq(\HDLCS,\EE_{\operatorname{top},\,\HDLCS})$ is extension closed, and thus exact in a natural way, this is not the case for other examples arising naturally in functional analysis. In particular, neither the quasi-complete, nor the sequentially complete, nor the locally complete spaces form an extension-closed subcategory of $(\HDLCS,\EE_{\operatorname{top},\,\HDLCS})$, see \cite[Table on p.~2114]{DS12}. The following result shows that it is nevertheless possible to do homological algebra with respect to the short topologically exact sequences in each of these three categories.

\begin{thm}\label{HCOM-DEX} The category $(\HCOM,\DD_{\operatorname{max},\,\HCOM})$ is strongly deflation-exact and the conflation structure $\DD_{\mathsf{max},\,\HCOM}=\EE_{\operatorname{top},\,\HDLCS}\cap\HCOM$ consists precisely of topologically exact sequences of locally convex spaces in which all three terms are hypo-complete. The embedding $(\HCOM,\DD_{\operatorname{max},\,\HCOM})\rightarrow\mathcal{LH}(\HCOM,\DD_{\operatorname{max},\,\HCOM})$ lifts to a bounded derived equivalence and the left heart is given by $\hMork(\HCOM)[\mathcal{S}_{\DD_{\operatorname{max},\,\HCOM}}^{-1}]$ as described in Section \ref{SEC-LRH}.
\end{thm}
\begin{proof} By Theorem \ref{THMDEFMAX} we have that $\DD_{\mathsf{max},\,\HCOM}$ is strongly deflation-exact and consists of kernel-cokernel pairs $(f,g)$ in which $g$ is a semistable cokernel\,---\,and thus by Propositions \ref{PROP-3} and \ref{PROP-6} an open surjection. Invoking Proposition \ref{PROP-KK} yields $(f,g)\in\EE_{\operatorname{top},\,\HDLCS}$. If conversely $(f,g)\in\EE_{\operatorname{top},\,\HDLCS}\cap\HCOM$ then $g$ is an open surjection and thus by Propositions \ref{PROP-3} and \ref{PROP-6} a $\HCOM$-semistable cokernel. The statements about the left heart follow from Theorems \ref{THM-LH-PHI} and \ref{NEW-LH}.
\end{proof}

\begin{rmk} We add the following comments.\vspace{2pt}

\begin{myitemize}

\item[(i)] The functor $\inc\colon(\HCOM,\DD_{\operatorname{max},\,\HCOM})\rightarrow(\HDLCS,\EE_{\operatorname{top},\,\HDLCS})$ is exact.

\vspace{3pt}

\item[(ii)] Given two hypo-complete spaces $E\subseteq F$ with continuous inclusion $i\colon E\rightarrow F$ such that $E$ is not $\tau_F$-closed, their quotient in the left heart is the object $0\hookrightarrow E\xrightarrow{\hspace{1pt}\scriptscriptstyle i\hspace{2pt}}F$. A morphism represented by a commutative square between such quotients is an isomorphism iff the corresponding diagram is a $\DD_{\operatorname{max},\,\HCOM}$-pulation, cf.~Remark \ref{MS-RMK}(iv). 

\vspace{3pt}

\item[(iii)]  The maximal inflation-exact structure is strongly inflation-exact and given by\vspace{-4pt}
\begin{align*}
\II_{\mathsf{max},\,\HCOM}&=\bigl\{E\xrightarrow{\hspace{2pt}f\hspace{2pt}}F\xrightarrow{\hspace{2pt}g\hspace{2pt}}G\in\HCOM\:\big|\:f \text{ is injective, open onto its range, } f(E)=g^{-1}(0)\\[-3pt]
& \hspace{22pt}\text{holds as vector spaces, } g \text{ is open onto its range, we have }G=\qcomp{g(F)}\text{ and} \\[-5pt]
& \hspace{22pt}\forall\:t\colon E\rightarrow T\in\HCOM\colon T\rightarrow\qcomp{(F\oplus T)/\overline{\ran[f\,\sm\hspace{-2pt}t]^{\T}}}\text{ has a closed range}\hspace{0.25pt}\bigr\}.\vspace{-5pt}
\end{align*}
Notice that in the last line above we know that the range of the canonical map is closed in $(F\oplus T)/\overline{\ran[f\,\sm\hspace{-2pt}t]^{\T}}$, since $f$ is an $\HDLCS$-kernel and thus $\HDLCS$-semistable. There seems to be no `operational' criterion for this closedness to be preserved under inclusion into the hypo-completion. The maximal exact structure is given by\vspace{-4pt}
\begin{align*}
\EE_{\mathsf{max},\,\HCOM}&=\bigl\{E\xrightarrow{\hspace{2pt}f\hspace{2pt}}F\xrightarrow{\hspace{2pt}g\hspace{2pt}}G\in\HCOM\:\big|\:(f,g) \text{ is is topologically exact and}\\[-3pt]
& \hspace{22pt}\forall\:t\colon E\rightarrow T\in\HCOM\colon (F\oplus T)/\overline{\ran[f\,\sm\hspace{-2pt}t]^{\T}}\text{ is hypo-complete}\hspace{0.25pt}\bigr\}.\vspace{-5pt}
\end{align*}

\vspace{3pt}

\item[(iv)] Using $\II_{\operatorname{max},\,\HCOM}$ we get an embedding into the right heart as in Section \ref{SEC-LRH}, using $\EE_{\operatorname{max},\,\HCOM}$ we can embed into the left or the right heart.

\vspace{3pt}

\item[(v)] Both of the conflation categories $(\HCOM,\DD_{\operatorname{max},\,\HCOM})$ and $(\HCOM,\EE_{\operatorname{max},\,\HCOM})$ do not have admissible (co-)kernels. Accordingly neither Theorem \ref{THM-HEART-2} nor Remark \ref{INF-RMK}(vi) are applicable. Whether $(\HCOM,\II_{\operatorname{max},\,\HCOM})$ does have admissible cokernels is unknown.

\vspace{3pt}

\item[(vi)] Besides the two trivial inclusions between the three conflation structures there is nothing else known to be true. In particular we do not know if any of the two inclusions is proper.
\end{myitemize}
\end{rmk}

To conclude this section we mention that the category of locally complete LB-spaces has been investigated in \cite{Wegner25} and that here, the maximal (two-sided) exact structure is strictly larger than the class of topologically exact sequences. Note that in the subcategory kernels compute differently than in $\HDLCS$.





\section{Fr\'echet Spaces of Type $(\operatorname{DN})$}\label{SEC:DN} 





While our abstract results in Section \ref{SEC-LRH} require \emph{either} the existence of all kernels \emph{or} the existence of all cokernels, our examples have so far all been preabelian. Below we provide an example in which all kernels exist but in which not every morphism has a cokernel.

\smallskip

Let $\FreDN\subseteq\Fre$ be the full subcategory of Fr\'echet spaces satisfying the property $(\operatorname{DN})$, see \cite[p.~359]{MV}. It is straightforward to check that $\FreDN$ is additive. By \cite[Lem 29.2(2)]{MV} $\FreDN$ has kernels and they are calculated as in $\Fre$. However, $\FreDN$ is not preabelian:

\smallskip

\begin{ex}(joint with J.~Wengenroth) Let $E$ be any Fr\'echet space with $(\operatorname{DN})$ which is not a Banach space, e.g., the space of rapidly decreasing sequences $E=\operatorname{s}$, cf.~\cite[Ex 29.4(1) and Lem 29.2(3)]{MV}. By Eidelheit's theorem, see \cite[Cor 26.28]{MV}, there exists a closed subspace $F\subseteq E$ such that $E/F\cong\omega$, i.e., the Fr\'echet space of all sequences. Let $f\colon F\rightarrow E$ be the inclusion map. By \cite[Lem 29.2(2)]{MV} the space $F$ has $(\operatorname{DN})$. Assume $g\colon E\rightarrow G$ is a cokernel of $f$ in the category $\FreDN$. Let $q\colon E\rightarrow\omega$ be the quotient map followed by the isomorphism $E/F\cong\omega$. Then $q=\cok f$ holds in the category $\Fre$ and we get a linear and continuous map $h\colon\omega\rightarrow G$ such that $hq=g$. As $G$ has $(\operatorname{DN})$, there exists a continuous norm $\|\cdot\|$ on $G$ and thus an $N\in\NN$ and a constant $C>0$ such that
$$
\forall\:x=(x_j)_{j\in\NN}\in\omega\colon \|h(x)\|\hspace{2pt}\leqslant \hspace{2pt}C\cdot\hspace{-3pt}\max_{j=1,\dots,N}|x_j|.
$$
It follows $\{x\in\omega\:|\:x_1=x_2=\cdots=x_N=0\}\subseteq h^{-1}(0)$ and thus $h(\omega)$ is of finite dimension. 
Let $n:=\dim h(\omega)$ and denote by $p\colon\omega\rightarrow\KK^{n+1}$ the projection onto the first $n+1$ components. By the cokernel property of $g$ in $\FreDN$ there exists a map $k\colon G\rightarrow\KK^{n+1}$ with $kg=pq$:
$$
\begin{tikzcd}[row sep=0.5cm]
& & & G\ar[rrdd, swap,"k"']\\
F\ar[rrru, bend left = 18, swap, "0"']\ar[rrrd, bend right = 18, "0"']\ar[rr, swap, "f"']&&E\ar[ur, swap,"g"']\ar[dr, "q"']& \\
& & & \omega\ar[uu, "h"']\ar[rr, "p"']&& \KK^{n+1}.
\end{tikzcd}
$$
It follows $khq = kg = pq$ and, since $q$ is surjective, we obtain $kh=p$. But this means $n+1=\dim p(\omega)=\dim k(h(\omega))\leqslant n$.\hfill\diam
\end{ex}

We emphasize that $(\FreDN,\DD_{\operatorname{max},\,\FreDN})$ is strongly deflation-exact with kernels and thus the corresponding results of Section \ref{SEC-LRH} hold true. In particular $\mathcal{LH}(\FreDN,\DD_{\operatorname{max},\,\FreDN})$ is bounded derived equivalent to $(\FreDN,\DD_{\operatorname{max},\,\FreDN})$ and admits an explicit description through 3-term complexes.


\section{Bornological Modules}\label{SEC:BOR}

Bornological categories have been investigated since the 1970s. Here we follow \cite[Section 2]{Meyer3} and focus on categories of bornological left modules over a $\KK$-algebra $A$, where $\KK$ denotes the complex numbers, and we tacitly assume that bornologies are complete and convex and algebras are unital. We mention however that generalizations of this setting are possible and appear in the literature; in particular one can work over other fields or drop/replace unitality.

\smallskip

Given a bornological $\KK$-algebra $A$ let $\BMod(A)$ denote the category of bornological left modules over $A$. The latter is a quasiabelian category \cite[Section 2.1--2.2]{Meyer3} and one can consider its left or right heart as in Theorem \ref{THM-HEART-2}. This has been done, e.g., in \cite{KKM, Savage, BK24, Bode} in the context of derived geometry.

\smallskip

For other applications, in particular in Hochschild cohomology, using all kernel-cokernel pairs as exact sequences on the quasiabelian category $\BMod(A)$ has certain disadvantages \cite[p.~3]{Meyer3}.  In order to rectify this, \cite[Section 2.4]{Meyer3} endows $\BMod(A)$ with the \emph{linearly split exact structure}
\begin{align*}
\EE_{\operatorname{lin\,split},\,\BMod(A)}&=\bigl\{E\xrightarrow{\hspace{2pt}f\hspace{2pt}}F\xrightarrow{\hspace{2pt}g\hspace{2pt}}G\;\big|\;(f,g) \text{ is a kernel-cokernel pair in }\BMod(A)\\[-2pt]
& \hspace{101pt}\exists\:s\colon G\rightarrow F \text{ bounded and $\KK$-linear}\colon gs=\id_G\bigr\}\vspace{-5pt}
\end{align*}
instead, i.e., the conflations are those short exact sequences of bornological modules that split when considered as short exact sequences of bornological vector spaces. This choice makes $(\BMod(A),\EE_{\operatorname{lin\,split},\,\BMod(A)})$ two-sided exact and guarantees that there are enough projectives and enough injectives \cite[Prop 1 in \S\,2.4 and Prop 3 und \S\,2.5]{Meyer3}. The linearly split exact structure has also been used in \cite{OS09, PPTT11, KS25}.

\begin{thm}\label{BMOD-THM} The category $(\BMod(A),\EE_{\operatorname{lin\,split},\,\BMod(A)})$ can be embedded likewise into its left or right heart, the embedding lifts to a bounded derived equivalence, the hearts are given by categories of 3-term complexes as described in Section \ref{SEC-LRH}.\diam{}
\end{thm}

\begin{rmk} Consider two bornological modules such that $E\subseteq F$ is algebraically a submodule with bounded inclusion map $i\colon E\rightarrow F$ but possibly not (bornologically) closed. Then we may consider the object $0\hookrightarrow E\xrightarrow{\hspace{1pt}\scriptscriptstyle i\hspace{1pt}}F$ in $\mathcal{LH}(\BMod(A),\EE_{\operatorname{lin\,split},\,\BMod(A)})$ as a formal quotient. A map given by a commutative diagram between two such quotients will be an isomorphism in the left heart iff the corresponding diagram is an $\EE_{\operatorname{lin\,split},\,\BMod(A)}$-pulation.\diam{}
\end{rmk}

In \cite[Section 5]{Meyer3}, a morphism of bornological algebras $i\colon A\rightarrow B$ is called \emph{isocohomological} if it induces a fully faithful functor $i^{*}\colon\Der(\BMod(B))\rightarrow\Der(\BMod(A))$, where both derived categories are taken with respect to the linearly split exact sequences. Assuming that the unbounded derived categories are well-defined, this is equivalent to $i$ inducing a fully faithful functor between the derived categories of the (abelian!) hearts.




\appendix

\section{The Left Canonical T-Structure}\label{SEC-APP}

Below we give a streamlined construction of the left canonical t-structure on $\Der(\mathcal{A},\CC)$ for a strongly deflation-exact category $(\mathcal{A},\CC)$ with kernels, employing our definition of acyclicity given in Definition \ref{DFN-AC-N}. We first recall the following.

\begin{prop}\label{TRIAG-K}\emph{(\cite[Prop 3.2]{HKRW}, see \cite[Prop 3.13]{HR19} for a proof)} Let $\mathcal{A}$ be an additive category with kernels. Then $\tau_{\LL}^{\leqslant n}C^{\bullet}\rightarrow{}C^{\bullet}\rightarrow\tau^{\geqslant n+1}_{\LL}C^{\bullet}\rightarrow\Sigma(\tau^{\leqslant n}_{\LL}C^{\bullet})$ is a dis\-tin\-guish\-ed triangle in $\K(\mathcal{A})$ for any $C^{\bullet}\in\C(\mathcal{A})$ and any $n\in\ZZ$.\hfill\diam
\end{prop}

If $\mathcal{A}$ has kernels, the truncation functors from \eqref{TRUNC} on the homotopy category induce a t-structure on $\K(\mathcal{A})$ as follows. The proof is a straightforward verification of the t-structure axioms and the equalities given below which we will omit.

\begin{prop}\label{K-T-STR}\cite[text after Prop 3.2]{HKRW} Let $\mathcal{A}$ be an additive category with kernels. Then $(\K^{\leqslant0}_L(\mathcal{A}),\K^{\geqslant0}_L(\mathcal{A}))$ with
\begin{eqnarray*}
\K^{\leqslant0}_{\LL}(\mathcal{A}) &\hspace{-7pt}=&\hspace{-7pt}\bigl\{C^{\bullet}\in\K(\mathcal{A})\:|\:\tau_{\LL}^{\geqslant1}C^{\bullet}\cong0\bigr\}\\[-3pt]
&\hspace{-7pt}= &\hspace{-7pt}\bigl\{C^{\bullet}\in\K(\mathcal{A})\:|\:\forall\:n\geqslant 1\colon\ker d^{\hspace{1pt}n-1}\stackrel{i^{n-2}}{\longrightarrow}C^{n-1}\stackrel{p^{n-1}}{\longrightarrow}\ker d^{\hspace{1pt}n}\text{ is a split ker-cok pair}\bigr\},\\[5pt]
\K^{\geqslant0}_{\LL}(\mathcal{A}) &\hspace{-7pt}= &\hspace{-7pt}\bigl\{C^{\bullet}\in\K(\mathcal{A})\:|\:\tau_{\LL}^{\leqslant-1}C^{\bullet}\cong0\bigr\}\\[-3pt]
&\hspace{-7pt}= &\hspace{-7pt}\bigl\{C^{\bullet}\in\K(\mathcal{A})\:|\:\forall\:n\leqslant-1\colon\ker d^{\hspace{1pt}n-1}\stackrel{i^{n-2}}{\longrightarrow}C^{n-1}\stackrel{p^{n-1}}{\longrightarrow}\ker d^{\hspace{1pt}n}\text{ is a split ker-cok pair}\bigr\}
\end{eqnarray*}
is a t-structure on $\K(\mathcal{A})$.\hfill\diam
\end{prop}

Using Definition \ref{DFN-AC-N} the following can be proven very easily.

\begin{lem}\label{EQUIV-LEM} Let $(\mathcal{A},\CC)$ be a strongly deflation-exact category with kernels and let $X^{\bullet}$ and $Y^{\bullet}$ be complexes over $\mathcal{A}$ satisfying $X^{\bullet}\cong Y^{\bullet}$ in $\K(\mathcal{A})$. Let $n\in\ZZ$ be fixed. Then $X^{\bullet}$ is acyclic in degree $n$ iff $Y^{\bullet}$ is acyclic in degree $n$.
\end{lem}
\begin{proof} As $X^{\bullet}\cong Y^{\bullet}$ holds in $\K(\mathcal{A})$ we get $\tau^{\leqslant n}_{\LL}\tau_{\LL}^{\geqslant n}(X^{\bullet})\cong\tau^{\leqslant n}_{\LL}\tau_{\LL}^{\geqslant n}(Y^{\bullet})$ in $\K(\mathcal{A})$ meaning that the complexes
\[
\cdots\rightarrow0\rightarrow\ker d_X^{\hspace{1pt}n-1}\stackrel{i^{n-2}_X}{\longrightarrow}X^{n-1}\stackrel{p^{n-1}_X}{\longrightarrow}\ker d_X^{\hspace{1pt}n}\rightarrow0\rightarrow\cdots
\]
and
\[
\cdots\rightarrow0\rightarrow\ker d_Y^{\hspace{1pt}n-1}\stackrel{i^{n-2}_Y}{\longrightarrow}Y^{n-1}\stackrel{p^{n-1}_Y}{\longrightarrow}\ker d_Y^{\hspace{1pt}n}\rightarrow0\rightarrow\cdots
\]
are isomorphic in $\K(\mathcal{A})$. The stated equivalence follows now from Theorem \ref{RMK-DC}(iii) and  $\Ac(\mathcal{A},\CC)$ being closed under isomorphisms in $\K(\mathcal{A})$ since the underlying category is strongly deflation-exact and karoubian, as we noted in Definition \ref{ACYCLIC}(ii).
\end{proof}

The standard method to get a t-structure on a quotient category is by the following result.

\begin{prop}\label{SCH-LEM}\emph{(\cite[Prop 3.3]{HKRW}, see \cite[Lem 1.2.17]{S} for a proof)} Let $(\mathcal{T}^{\leqslant0},\mathcal{T}^{\geqslant0})$ be a t-structure on a triangulated category $\mathcal{T}$ and let $\mathcal{N}$ be a thick subcategory. Write $Q\colon\mathcal{T}\rightarrow\mathcal{T}/\mathcal{N}$ for the corresponding quotient. The pair $(Q(\mathcal{T}^{\leqslant0}),Q(\mathcal{T}^{\geqslant0}))$ is a t-structure on $\mathcal{T}/\mathcal{N}$ iff for every distinguished triangle $X_1\rightarrow X_0\rightarrow N\rightarrow \Sigma X_1$ with $X_1\in\mathcal{T}^{\geqslant1}$, $X_0\in\mathcal{T}^{\leqslant0}$ and $N\in\mathcal{N}$ we have $X_1, X_0\in\mathcal{N}$.\hfill\diam
\end{prop}

The next result is the key observation necessary in order to verify the criterion in Proposition \ref{SCH-LEM} for our concrete situation. Notice the comments on its origin in Remark \ref{FIN-RMK}.

\begin{lem}\label{NEW-PROP} Let $(\mathcal{A},\CC)$ be a strongly deflation-exact category with kernels. Consider a distinguished triangle $X^{\bullet}\xrightarrow{\scriptscriptstyle f}Y^{\bullet}\xrightarrow{\scriptscriptstyle g}Z^{\bullet}\xrightarrow{\scriptscriptstyle h}\Sigma X^{\bullet}$ in $\K(\mathcal{A})$ and let $n\in\ZZ$ be fixed. If $X^{\bullet}$ and $Z^{\bullet}$ are acyclic in degree $n$, then the same is true for $Y^{\bullet}$.
\end{lem}
\begin{proof} By rotation we get the distinguished triangle $\Sigma^{-1}Z^{\bullet}\xrightarrow{\scriptscriptstyle -\Sigma^{-1}h} X^{\bullet}\xrightarrow{\scriptscriptstyle f}Y^{\bullet}\xrightarrow{\scriptscriptstyle g}Z^{\bullet}$ for which $Y^{\bullet}\cong\cone(-\Sigma^{-1}h)$ holds in $\K(\mathcal{A})$. By Lemma \ref{EQUIV-LEM} it suffices to show that $\cone(-\Sigma^{-1}h)$ is acyclic in degree $n$. We therefore assume $Y^{\bullet}=\cone(-\Sigma^{-1}h)$ and get the commutative diagram
\begin{equation}\label{EQS-0}
\scalebox{0.96}{\begin{tikzcd}[
  ampersand replacement=\&,
  column sep=0pt,
  row sep=2.5em
]
|[inner xsep=0.12cm]|\Sigma^{-1}Z^{\bullet}
  \ar[d,swap, "-\Sigma^{-1}h"']
\&
|[inner xsep=0.12cm]|\cdots
  \ar[rr]
\&
|[inner xsep=0pt]|\hspace{0.47cm}
\&
Z^{n-2}
  \ar[rr, "-d_Z^{\hspace{1pt}n-2}"]
  \ar[d, "-h^{n-2}"']
\&
|[inner xsep=0pt]|\hspace{2.45cm}
\&
Z^{n-1}
  \ar[rr, "-d_Z^{\hspace{1pt}n-1}"]
  \ar[d, "-h^{n-1}"']
\&
|[inner xsep=0pt]|\hspace{1.65cm}
\&
Z^n
  \ar[rr]
  \ar[d, "-h^n"']
\&
|[inner xsep=0pt]|\hspace{0.47cm}
\&
\cdots
\\
|[inner xsep=0.12cm]|X^{\bullet}
  \ar[d,swap, "f"']
\&
|[inner xsep=0.12cm]|\cdots
  \ar[rr]
\&
|[inner xsep=0pt]|\hspace{0.47cm}
\&
X^{n-1}
  \ar[rr, "d_X^{\hspace{1pt}n-1}"]
  \ar[d]
\&
|[inner xsep=0pt]|\hspace{2.45cm}
\&
X^n
  \ar[rr, "d_X^{\hspace{1pt}n}"]
  \ar[d]
\&
|[inner xsep=0pt]|\hspace{1.65cm}
\&
X^{n+1}
  \ar[rr]
  \ar[d]
\&
|[inner xsep=0pt]|\hspace{0.47cm}
\&
\cdots
\\
|[inner xsep=0.12cm]|Y^{\bullet}
\&
|[inner xsep=0.12cm]|\cdots
  \ar[rr]
\&
|[inner xsep=0pt]|\hspace{0.47cm}
\&
Z^{n-1}\oplus X^{n-1}
  \ar[
    rr,
    "d_Y^{\hspace{1pt}n-1}",
    "{\left[
      \begin{smallmatrix}
        \phantom{-}d_Z^{\hspace{1pt}n-1} & 0 \\
        -h^{n-1} & d_X^{\hspace{1pt}n-1}
      \end{smallmatrix}
    \right]}"'
  ]
\&
|[inner xsep=0pt]|\hspace{2.45cm}
\&
Z^n\oplus X^n
  \ar[
    rr,
    "d_Y^{\hspace{1pt}n}",
    "{\left[
      \begin{smallmatrix}
        \phantom{-}d_Z^{\hspace{1pt}n} & 0 \\
        -h^n & d_X^{\hspace{1pt}n}
      \end{smallmatrix}
    \right]}"'
  ]
\&
|[inner xsep=0pt]|\hspace{1.65cm}
\&
Z^{n+1}\oplus X^{n+1}
  \ar[rr]
\&
|[inner xsep=0pt]|\hspace{0.47cm}
\&
\cdots
\end{tikzcd}}
\end{equation}
with $Z^{n-1}$ in degree $n$. Denote by $\bigl[\begin{smallmatrix}i_1\\i_2\end{smallmatrix}\bigr]=i_{Y}^{\hspace{0.5pt}n-1}\colon\ker d_Y^{\hspace{1pt}n}\hookrightarrow Z^{n}\oplus X^{n}$ the kernel of $d_Y^{\hspace{1pt}n}$ and by $i_X^{\hspace{0.5pt}n-1}\colon\ker d_X^{\hspace{1pt}n}\hookrightarrow X^n$ the kernel of $d_X^{\hspace{1pt}n}$. Using the diagram above we see that 
$$
\Bigl[\begin{smallmatrix}\phantom{-}d_Z^{\hspace{1pt}n-1}&0\\-h^{n-1}&i_X^{\hspace{0.5pt}n-1}\end{smallmatrix}\Bigr]\colon Z^{n-1}\oplus\ker d_{X}^{\hspace{1pt}n}\rightarrow Z^{n}\oplus X^n
$$
factors uniquely through the kernel $i_Y^{n-1}$ via a map $[\delta_1\;\delta_2]\colon Z^{n-1}\oplus\ker d_X^{\hspace{1pt}n}\rightarrow \ker d_Y^{\hspace{1pt}n}$. Explicitly calculating the latter as well as $d_Y^{\hspace{1pt}n}i_Y^{\hspace{0.5pt}n-1}=0$ gives the equations
\begin{equation}\label{EQS-1}
i_1\delta_1=d_Z^{\hspace{1pt}n-1},\; i_2\delta_1=-h^{n-1},\;i_1\delta_2=0,\;i_2\delta_2=i_X^{n-1},\;d_Z^{\hspace{1pt}n}i_1=0,\;d_X^{\hspace{1pt}n}i_2-h^ni_1=0.
\end{equation}
Considering the diagram
\[
\begin{tikzcd}[ampersand replacement=\&, row sep=20pt]
\ker d_Y^{\hspace{1pt}n}\ar[d,dashed,swap,"\exists!\:q"]\ar[r, "{\bigl[\begin{smallmatrix}i_1 \\ i_2\end{smallmatrix}\bigr]}"]\& Z^n\oplus X^n\ar[d,"{[1\;0]}"]\ar[r,"d_Y^{\hspace{1pt}n}"]\&Z^{n+1}\oplus X^{n+1}\ar[d,"{[1\;0]}"]\\
\ker d_Z^{\hspace{1pt}n}\ar[r,swap,"i_Z^{n-1}"]\&Z^n\ar[r,swap,"d_Z^{\hspace{1pt}n}"]\&Z^{n+1}
\end{tikzcd}
\]
we get a map $q$ with\vspace{-8pt}
\begin{equation}\label{EQS-2}
i_Z^{\hspace{0.5pt}n-1}q=i_1.
\end{equation}
Using \eqref{EQS-1} and \eqref{EQS-2} plus the fact that $i_Z^{\hspace{1pt}n-1}$ is monic we get the commutative diagram
\begin{equation}\label{PB-PRF}
\begin{tikzcd}[column sep = 33pt]
Z^{n-1}\oplus\ker d_X^{\hspace{1pt}n}\ar[r,"{[\delta_1\;\delta_2]}"]\ar[d,swap,"{[1\;0]}"]&\ker d_Y^{\hspace{1pt}n}\ar[d,"q"]\\
Z^{n-1}\ar[r,->>,swap,"p_Z^{n-1}"]&\ker d_Z^{\hspace{1pt}n}
\end{tikzcd}
\end{equation}
in which $p_Z^{n-1}$ is a deflation as $Z^{\bullet}$ is acyclic in degree $n$ by assumption. We claim that \eqref{PB-PRF} is a pullback. Given $t\colon T\rightarrow \ker d_Y^{\hspace{1pt}n}$ and $s\colon T\rightarrow Z^{n-1}$ with $qt=p_Z^{n-1}s$ we compute $d_X^{\hspace{1pt}n}i_2(t-\delta_1s)=0$ using \eqref{EQS-0}--\eqref{EQS-2}. The universal property of the kernel $i_X^{n-1}\colon\ker d_X^{\hspace{1pt}n}\hookrightarrow X^n$ thus yields a unique $\alpha\colon T\rightarrow\ker d_X^{\hspace{1pt}n}$ with $i_2(t-\delta_1s)=i_X^{n-1}\alpha$. Another computation shows that $\bigl[\begin{smallmatrix}s\\\alpha\end{smallmatrix}\bigr]\colon T\rightarrow Z^{n-1}\oplus\ker d_X^{\hspace{1pt}n}$ is the unique map with $[\delta_1\;\delta_2]\bigl[\begin{smallmatrix}s\\\alpha\end{smallmatrix}\bigr]=t$ and $[1\;0]\bigl[\begin{smallmatrix}s\\\alpha\end{smallmatrix}\bigr]=s$ which establishes that \eqref{PB-PRF} is a pullback. By \hypref{R2} the map $[\delta_1\;\delta_2]$ is thus a deflation, too. By assumption $p_X^{n-1}\colon X^{n-1}\twoheadrightarrow\ker d_X^{\hspace{1pt}n}$ is a deflation and, again by \hypref{R2}, so is its pullback
\begin{equation*}
\begin{tikzcd}[ampersand replacement=\&, column sep=45pt]
Z^{n-1}\oplus X^{n-1}\ar[r,->>,"{\bigl[\begin{smallmatrix}1&0\\0&p_X^{n-1}\end{smallmatrix}\bigr]}"]\ar[d,swap,"{[0\;1]}"]\&Z^{n-1}\oplus\ker d_X^{\hspace{1pt}n}\ar[d,"{[0\;1]}"]\\
X^{n-1}\ar[r,->>,swap,"p_X^{n-1}"]\&\ker d_X^{\hspace{1pt}n}.
\end{tikzcd}
\end{equation*}
By \hypref{R1} we obtain that $[\delta_1\;\delta_2]\bigl[\begin{smallmatrix}1&0\\0&p_X^{n-1}\end{smallmatrix}\bigr]=[\delta_1\;\delta_2p_X^{n-1}]$ is a deflation. Finally, we observe that the latter coincides with $p_Y^{n-1}$ and that $i_Y^{\hspace{0.5pt}n-2}$ is its kernel
\[
\begin{tikzcd}[ampersand replacement=\&, column sep=18pt, row sep =25pt]
\ker d_Y^{\hspace{1pt}n-1}\ar[rd,>->,"i_Y^{n-2}"] \& \& \& \\
\& Z^{n-1}\oplus X^{n-1}\ar[rd,->>, "p_Y^{n-1}", "{[\delta_1\;\delta_2p_X^{n-1}]}"']
\ar[
    rr,
    "{\left[
      \begin{smallmatrix}
        \phantom{-}d_Z^{\hspace{1pt}n-1} & 0 \\
        -h^{n-1} & d_X^{\hspace{1pt}n-1}
      \end{smallmatrix}
    \right]}",
    "d_Y^{\hspace{1pt}n-1}"'
  ]
\& \&Z^n\oplus X^n \& \\
\& \&\ker d_Y^{\hspace{1pt}n}\ar[ru,hook,swap,"i_Y^{\hspace{0.5pt}n-1}"] \&
\end{tikzcd}
\]
showing that $Y^{\bullet}=\cone(-\Sigma^{-1}h)$ is acyclic in degree $n$.
\end{proof}

We can now prove the assertion announced at the beginning of Section \ref{SEC-LRH}.

\begin{thm}\label{L-T-STRC}\cite[Prop 3.5]{HKRW} Let $(\mathcal{A},\CC)$ be a strongly deflation-exact category with kernels. Then the pair $(\Der^{\leqslant0}_{\LL}(\mathcal{A},\CC),\Der^{\geqslant 0}_{\LL}(\mathcal{A},\CC))$ as defined in \eqref{LEFT-T-S} is a t-structure on $\Der(\mathcal{A},\CC)$.
\end{thm}
\begin{proof} The equivalence of the two descriptions of $\Der^{\leqslant0}_{\LL}(\mathcal{A},\CC)$ and $\Der^{\geqslant 0}_{\LL}(\mathcal{A},\CC)$ can be checked by straightforward computations that we will omit here.

\smallskip

To establish the t-structure property, we apply Proposition \ref{SCH-LEM} to $(\K_{\LL}^{\leqslant0}(\mathcal{A}),\K_{\LL}^{\geqslant0}(\mathcal{A}))$. To this end let a distinguished triangle $X_1^{\bullet}\rightarrow X_0^{\bullet}\rightarrow N^{\bullet}\rightarrow \Sigma X_1^{\bullet}$ in $\K(\mathcal{A})$ be given with $X_1^{\bullet}\in\K^{\geqslant1}_{\LL}(\mathcal{A})=\Sigma^{-1}\K_{\LL}^{\geqslant0}(\mathcal{A})$, $X_0^{\bullet}\in\K^{\leqslant0}_{\LL}(\mathcal{A})$ and $N^{\bullet}\in\Ac(\mathcal{A},\CC)$. By Proposition \ref{K-T-STR} we get that $X_1^{\bullet}$ is acyclic in every degree $n\leqslant0$ and that $X_0^{\bullet}$ is acyclic in every degree $n\geqslant1$. Since $N^{\bullet}$ is acyclic in every degree, Lemma \ref{NEW-PROP} yields that $X_0^{\bullet}$ is also acyclic in every degree $n\leqslant0$, i.e., $X_0^{\bullet}\in\Ac(\mathcal{A},\CC)$. Since $\Ac(\mathcal{A},\CC)\subseteq\K(\mathcal{A})$ is a triangulated subcategory, cf.\ Definition \ref{ACYCLIC}(ii), $X_1^{\bullet}\in\Ac(\mathcal{A},\CC)$, too.
\end{proof}

\begin{rmk}\label{FIN-RMK} We conclude with the following pointers.\vspace{2pt}

\begin{myitemize}

\item[(i)] Lemma \ref{NEW-PROP} and Theorem \ref{L-T-STRC} are almost identical to \cite[Prop 1.2.14 and consequences noted after that]{S}, where however a quasiabelian category endowed with its maximal exact structure consisting of all kernel-cokernel pairs was considered. That these statements hold under somewhat weaker assumptions was noticed in \cite{SW11Trier} but never published\,---\,one-sided exact categories were not considered in \cite{SW11Trier}.

\vspace{3pt}

\item[(ii)] Lemma \ref{NEW-PROP} implies: Let $f\colon X^{\bullet}\rightarrow Y^{\bullet}$ be a morphism in $\C(\mathcal{A})$ where $(\mathcal{A},\CC)$ is a strongly deflation-exact category with kernels and let $n\in\ZZ$. If $X^{\bullet}$ is acyclic in degree $n$ and $Y^{\bullet}$ is acyclic in degree $n-1$, then $\cone f$ is acyclic in degree $n-1$.

\vspace{3pt}

 \item[(iii)] Employing the notion of acyclicity in a single degree as in \cite[Dfn 2.14]{HKRW}, the statement of \cite[Prop 3.4]{HKRW}, which is an analogue of (ii) above, is not correct. Its proof however is correct but shows a slightly different statement which nevertheless suffices to prove \cite[Prop 3.5]{HKRW} which is exactly Theorem \ref{L-T-STRC} above.\hfill\diam
\end{myitemize}
\end{rmk}

\smallskip

{\begin{center}
{\sc Declaration on the Use of AI}
\end{center}\nopagebreak{}
During the preparation of this article the authors used ChatGPT (versions Sol and Astra) to assist with literature searches, explore proof strategies, discuss and check mathematical arguments and calculations, and improve the exposition, grammar and spelling. The authors critically evaluated and revised all chatbot-made suggestions and take full responsibility for the mathematical results, proofs, references and final text.}

\smallskip

{\begin{center}
{\sc Acknowledgments}
\end{center}\nopagebreak{}
The authors thank Adam-Christiaan van Roosmalen for several discussions concerning the content of the appendix. They also thank Joachim Hilgert for discussions that inspired the results in Section  \ref{SEC:HCOM}.}


\end{document}